\documentclass[a4paper,reqno]{amsart}
\usepackage{amsmath,amsfonts,amssymb,amsthm}
\usepackage{xy}
\xyoption{all}
\usepackage{enumitem}
\usepackage[colorlinks, linkcolor=blue!72,anchorcolor=orange,
    citecolor=red,urlcolor=Emerald, bookmarksopen,bookmarksdepth=2]{hyperref}
\usepackage[dvipsnames]{xcolor}
\usepackage{cleveref}
\usepackage{url}

\theoremstyle{plain}
\newtheorem{theorem}{Theorem}[section]
\newtheorem{lemma}[theorem]{Lemma}
\newtheorem{proposition}[theorem]{Proposition}
\newtheorem{corollary}[theorem]{Corollary}
\theoremstyle{definition}
\newtheorem{definition}[theorem]{Definition}
\newtheorem{remark}[theorem]{Remark}
\newtheorem{example}[theorem]{Example}

\numberwithin{equation}{section}

\usepackage[colorlinks, linkcolor=blue!72,anchorcolor=orange,citecolor=red,urlcolor=Emerald, bookmarksopen,bookmarksdepth=2]{hyperref}
\usepackage{cleveref}

\newcommand{\Hom}{\operatorname{Hom}}

\newcommand{\Ext}{\operatorname{Ext}}
\newcommand{\thick}{\operatorname{thick}}
\newcommand{\per}{\operatorname{per}}
\newcommand{\add}{\operatorname{add}}
\newcommand{\ladd}{\operatorname{Add}}
\newcommand{\proj}{\operatorname{proj}}
\newcommand{\lproj}{\operatorname{Proj}}

\renewcommand{\mod}{\operatorname{mod}}
\newcommand{\lmod}{\operatorname{Mod}}
\newcommand{\Fac}{\operatorname{Fac}}

\newcommand{\emod}[2]{#1\text{-}\mod #2}
\newcommand{\elmod}[2]{#1\text{-}\lmod #2}

\newcommand{\AR}[1]{\tau_{[#1]}}

\def\k{\mathbf{k}}
\def\A{\mathcal{A}}
\def\C{\mathcal{C}}
\def\D{\mathcal{D}}

\def\dH{d\text{-}\mathcal{H}}
\def\dplusH{(d+1)\text{-}\mathcal{H}}
\def\S{\mathcal{S}}
\def\M{\mathcal{M}}
\def\K{\mathcal{K}}
\def\T{\mathcal{T}}
\def\F{\mathcal{F}}
\def\H{\mathcal{H}}
\def\P{\mathcal{P}}
\def\I{\mathcal{I}}
\def\X{\mathcal{X}}
\def\Y{\mathcal{Y}}

\def\MM{M}
\def\NN{N}
\def\LL{L}
\def\SS{S}
\def\PP{P}

\def\WW{W}
\def\XX{X}
\def\YY{Y}
\def\ZZ{Z}
\def\TT{T}
\def\FF{F}

\title{AIR tilting subcategories of extended hearts}

\author{Jiaqun Wei}
\address{School of Mathematical Sciences, Zhejiang Normal University, 321004 Jinhua, China}
\email{weijiaqun5479@zjnu.edu.cn}

\author{Yu Zhou}
\address{School of Mathematical Sciences,
Beijing Normal University,
100875 Beijing,
China}
\email{yuzhoumath@gmail.com}



\thanks{The work was supported by the National Natural Science Foundation of China (Grant Nos. 12571042, 12271279, 12271249, 12031007) and by the Natural Science Foundation of Zhejiang Province (Grant No. LZ25A010002).}

\begin{document}

\begin{abstract}
    We introduce the notion of AIR tilting subcategories of extended hearts of $t$-structures on a triangulated category associated with silting subcategories. This notion generalizes $\AR{d}$-tilting pairs of extended finitely generated modules over finite-dimensional algebras to a more general framework, which includes both extended large modules over unitary rings and truncated subcategories of finite-dimensional derived categories of proper non-positive differential graded algebras. Within this setting, we establish a bijection between AIR tilting subcategories and silting subcategories. Furthermore, we define quasi-tilting and tilting subcategories of extended hearts, extending the corresponding notions from module categories, and investigate their fundamental properties along with the relationships among these tilting-related classes.
\end{abstract}

\maketitle

\section*{Introduction}

Tilting theory has long been one of the central tools in representation theory, providing powerful methods to induce equivalences between categories. These include, most notably, equivalences between subcategories of module categories as well as between derived categories. The subject originated in the late 1970s within the representation theory of finite-dimensional algebras. Inspired by the study of reflection functors \cite{BGP,APR}, the notion of (classical) tilting modules was introduced and developed in \cite{BB,HR}. This concept was subsequently generalized to modules of finite projective dimension, leading to the development of generalized tilting modules \cite{Mi}. These classical and generalized notions were later extended to module categories over arbitrary unitary rings, significantly expanding the scope of the theory to encompass modules that are not necessarily finitely generated \cite{ColbF,CT,AC,Ba}.

More recently, the development of cluster algebras has stimulated significant progress in the theory, notably with the introduction of support $\tau$-tilting modules for finite-dimensional algebras \cite{AIR}. A fundamental result in this direction is the bijection between support $\tau$-tilting modules, two-term silting complexes, and functorially finite torsion classes \cite{AIR,DF}. This framework has since been extended to the setting of arbitrary rings \cite{AMV}, where silting modules serve as the appropriate analogues. In the context of derived categories, silting complexes arise naturally in the classification of $t$-structures \cite{KV,KY}. To address the correspondence problem for silting complexes of arbitrary length, the category of $d$-extended finitely generated modules was introduced in \cite{G,Z}, where it was established that $\AR{d}$-tilting pairs correspond to $(d+1)$-term silting complexes \cite{Z}.

In this paper, for any positive integer $d$, we introduce the notion of AIR tilting subcategories within the $d$-extended hearts of $t$-structures associated with silting subcategories. This construction extends the aforementioned bijections to a significantly more general setting (see Theorem~\ref{thm:bi}). Our approach provides a unified framework that simultaneously recovers the cases of extended finitely generated module categories over finite-dimensional algebras, extended large module categories over arbitrary unitary rings, and truncated subcategories of finite-dimensional derived categories of proper non-positive differential graded algebras. We refer the reader to Section~\ref{sec:app} for further details and applications.

Within this unified framework, we also introduce the notion of quasi-tilting subcategories of extended hearts, generalizing the concept of quasi-tilting modules. The latter were originally introduced in \cite{CDT} to realize one half of the equivalences in the classical Brenner–Butler tilting theorem, and were subsequently extended to the setting of infinitely generated modules in \cite{AMV}. We investigate the homological and categorical properties of the associated $d$-factor subcategories (see Theorem~\ref{thm:qtilt}), and establish that every AIR tilting subcategory is quasi-tilting. For further connections between quasi-tilting modules and silting modules in the classical setting, we refer the reader to \cite{AMV,W4}.

Finally, we introduce and study tilting subcategories of extended hearts. These subcategories can be viewed as a generalization of classical tilting modules; in particular, they share the characteristic property that their induced torsion classes contain all injective objects (see Corollary~\ref{cor:injmap2}). They also generalize tilting modules of arbitrary finite projective dimension (see Example~\ref{exm:tiltmod}). We establish several characterizations of tilting subcategories in terms of their relationships with AIR tilting and quasi-tilting subcategories. For the precise statements, we refer the reader to Proposition~\ref{prop5} and Theorem~\ref{thm:equiv}.

The paper is organized as follows. In Section~\ref{sec:1}, we recall the basic definitions and properties of silting subcategories. In particular, we provide a criterion for a presilting subcategory to be silting, which plays a key role in our subsequent proofs. Section~\ref{sec:2} is devoted to the study of extended hearts of $t$-structures associated with silting subcategories. Here, we investigate the properties of torsion classes induced by silting subcategories and, building on these results, we introduce the notion of AIR tilting subcategories and establish their bijection with silting subcategories. In Section~\ref{sec:3}, we introduce and study quasi-tilting subcategories of extended hearts, focusing on the fundamental properties of their associated $d$-factor subcategories. Section~\ref{sec:tilt} is concerned with tilting subcategories of extended hearts and their relationships with AIR tilting and quasi-tilting subcategories. Finally, in Section~\ref{sec:app}, we apply our results to two specific settings, which recover the cases of extended finitely generated module categories over finite-dimensional algebras and extended large module categories over unitary rings.

\subsection*{Acknowledgment}

The second author is grateful to Jie Xiao for a very helpful suggestion that enabled the extension of an earlier version of this paper to a more general setting. He also wishes to thank Xiao-Wu Chen, Dong Yang, and Bin Zhu for their insightful discussions on related topics. Additionally, he would like to thank Jifen Liu and Liangwei Huang for pointing out the omission of certain necessary assumptions in a previous draft.

\section{Silting subcategories}\label{sec:1}

Throughout the paper, the shift functor in a triangulated category is denoted by $[1]$. For any subcategory $\X$ of an additive category $\C$, we denote by $\add\X$ the \emph{additive closure} of $\X$ in $\C$, that is, the full subcategory of $\C$ consisting of direct summands of finite direct sums of objects of $\X$. When we say that $\X$ is a subcategory of an additive category $\C$, denoted by $\X\subseteq\C$, we always assume that $\X$ is full and satisfies $\X=\add\X$. We write $\XX\in\C$ to indicate that $\XX$ is an object of $\C$. For any $X,Y\in\C$, we simply denote $\Hom_\C(\XX,\YY)$ by $\Hom(\XX,\YY)$ if no confusion can arise. 

For any $\X,\Y\subseteq\C$, we write $\Hom(\X,\Y)=0$ to mean that $\Hom(\XX,\YY)=0$ for all $\XX\in\X$ and $\YY\in\Y$. The analogous conventions apply to the notations $\Hom(\X,\YY)=0$ and $\Hom(\XX,\Y)=0$. A morphism $f\in\Hom(\XX,\YY)$ is called a \emph{right $\X$-approximation} of $\YY$, if $\XX\in\X$ and every morphism $g\in\Hom(\XX',\YY)$ with $\XX'\in\X$ factors through $f$. We say that $\X$ is a contravariantly finite subcategory of $\C$ if every object of $\C$ admits a right $\X$-approximation.

For any subcategory $\X$ of a triangulated category $\C$, we consider the following full subcategories of $\C$:
$$\C^{\leq 0}_\X:=\{\YY\in\C\mid \Hom(\X,\YY[i])=0\text{ for all $i>0$} \},$$
$$\C^{\geq 0}_\X:=\{\YY\in\C\mid \Hom(\X,\YY[i])=0\text{ for all $i<0$} \}.$$
For any two subcategories $\X$ and $\Y$ of $\C$, we denote by $\X\ast\Y$ the full subcategory of $\C$ consisting of all objects $\ZZ$ such that there exists a triangle
$$\XX\to \ZZ\to\YY\to\XX[1],$$
in $\C$ with $\XX\in\X$ and $\YY\in\Y$. A subcategory $\X$ of $\C$ is called \emph{closed under extensions} if $\X*\X\subseteq\X$.

Throughout the remainder of this section, let $\C$ be an arbitrary triangulated category. We recall the following notion from \cite{KV,AI}. 
\begin{definition}\label{def:silting}
    Let $\S$ be a subcategory of $\C$. 
    \begin{enumerate}
        \item We call $\S$ \emph{presilting} if 
        $\Hom(\S,\S[i])=0$ 
        for all $i>0$.
        \item We call $\S$ \emph{silting} if it is presilting and the smallest triangulated subcategory of $\C$ containing $\S$ is $\C$ itself.
    \end{enumerate}
    For two silting subcategories $\S$ and $\S'$ of $\C$, we write $\S\geq \S'$ if $\Hom(\S,\S'[i])=0$ for any $i>0$.
\end{definition}

Let $\S$ be a silting subcategory of $\C$. For any integers $p\leq q$, we denote 
$$\S^{[p,q]}:=\S[p]*\S[p+1]*\cdots*\S[q].$$
By \cite[Proposition~2.7]{IYa}, $\S^{[p,q]}$ is closed under direct summands. For any integer $d\geq 0$, an object (resp. subcategory) of $\S^{[0,d]}$ is called \emph{$(d+1)$-term} with respect to $\S$. We collect some basic properties of a silting subcategory.

\begin{lemma}[{\cite[Proposition~2.23]{AI}}]\label{AI2.23}
    Let $\S$ be a silting subcategory of $\C$. Then
    \[\C=\bigcup_{n\geq 0}\S^{[-n,n]}.\]
\end{lemma}

\begin{lemma}[{\cite[Lemma~3.7]{AMY}}]\label{compare}
    Let $\S$ and $\S'$ be silting subcategories of $\C$ and let $d\geq 0$ be an integer. Then $\S\subseteq\S'^{[0,d]}$ if and only if $\S'\subseteq\S^{[-d,0]}$.
\end{lemma}

\begin{lemma}[{\cite[Proposition~2.14 and Theorem~2.11]{AI}}]\label{sil=}
    For any two silting subcategories $\S$ and $\S'$ of $\C$, if $\C_\S^{\leq 0}=\C_{\S'}^{\leq 0}$, then $\S=\S'$.
\end{lemma}

A \emph{thick} subcategory of $\C$ is a triangulated subcategory closed under direct summands.

\begin{remark}\label{rmk:silting}
    Let $\S$ be a presilting subcategory of $\C$. Then $\S$ is silting if and only if the smallest thick subcategory of $\C$ containing $\S$ is $\C$ itself. Indeed, by Lemma~\ref{AI2.23}, $\C$ is the union of $\S^{[-n,n]}$, $n\geq 0$, each of which is closed under direct summands.
\end{remark}

The next lemma is a slight modification of \cite[Lemma~4.2]{AMY}, whose proof proceeds in essentially the same way.

\begin{lemma}\label{lem01}
    Let $\S$ be a contravariantly finite presilting subcategory of $\C$. Then for any integer $l>0$, we have
    \[\C_\S^{\leq 0}=\S^{[0,l-1]}*\C_\S^{\leq 0}[l].\]
\end{lemma}

Below, we provide a necessary and sufficient condition for a presilting subcategory to be silting, which will be used in the next section. We also refer the reader to \cite[Proposition~4.2]{W1} and \cite[Proposition~4.2]{AMV} for a similar criterion in the case $\C=K^b(\lproj R)$ the homotopy category of bounded complexes of projective modules over a unitary ring $R$.

\begin{proposition}\label{prop1}
    Let $\P$ be a silting subcategory of $\C$, and $d$ be a positive integer. A contravariantly finite presilting subcategory $\S$ of $\C$ that is $(d+1)$-term with respect to $\P$ is silting if and only if $\C_\S^{\leq 0}\subseteq \C_\P^{\leq 0}$.
\end{proposition}

\begin{proof}
    The ``only if" part: Since $\S\subseteq\P^{[0,d]}$, we have $\S\leq \P$. Hence, it follows directly from \cite[Proposition~2.14]{AI} that $\C_\S^{\leq 0}\subseteq \C_\P^{\leq 0}$. 
    
    The ``if" part: Since $\S\subseteq\P^{[0,d]}$, we have $\Hom(\S,\P[i])=0$ for any $i\geq d+1$. Thus, $\P[d]\subseteq \C_\S^{\leq 0}$. By Lemma~\ref{lem01}, we have
    $$\C_\S^{\leq 0}\subseteq\S^{[0,d]}*\C_\S^{\leq 0}[d+1].$$
    Therefore, for any $\PP\in\P$, there exists a triangle
    $$\XX\to \PP[d]\to\YY\to\XX[1],$$
    where $\XX\in\S^{[0,d]}$ and $\YY\in\C_\S^{\leq 0}[d+1]$. Then, by $\C_\S^{\leq 0}\subseteq \C_\P^{\leq 0}$, we have $\YY\in \C_\P^{\leq 0}[d+1]$. So $\Hom(\P[d],\YY)=0$, which, together with the above triangle, implies that $\PP[d]$ is a direct summand of $\XX$. Consequently, the smallest thick subcategory of $\C$ containing $\S$ includes $\P[d]$. Therefore, it includes $\P$, and therefore equals $\C$. Thus, by Remark~\ref{rmk:silting}, $\S$ is silting.
\end{proof}

\section{AIR tilting subcategories of extended hearts}\label{sec:2}

Throughout the rest of the paper, let $\D$ be a triangulated category admitting a presilting subcategory $\P$ satisfying the following conditions.
\begin{enumerate}[label=(P\arabic*), ref=(P\arabic*)]
    \item[(P0)] $\P$ is contravariantly finite in $\D$.
    \item\label{item:P1} There exists an integer $m<0$ such that $\Hom(\P,\P[i])=0$ for all $i<m$.
    \item\label{item:P2} $(\D^{\leq 0}_\P,\D^{\geq 0}_\P)$ is a non-degenerate $t$-structure on $\D$. 
\end{enumerate}
See Section~{sec:app} for motivating examples. For simplicity, we write $\D^{\leq 0}:=\D^{\leq 0}_\P$ and $\D^{\geq 0}:=\D^{\geq 0}_\P$. We then consider the following subcategories of $\D$:
\begin{itemize}
    \item $\D^{\leq q}:=\D^{\leq 0}[-q]$ and $\D^{\geq q}:=\D^{\geq 0}[-q]$ for any integer $q$;
    \item $\D^{[p,q]}:=\D^{\leq q}\cap\D^{\geq p}$ for any integers $p\leq q$;
    \item $\H:=\D^{[0,0]}=\D^{\leq 0}\cap\D^{\geq 0}$ the heart of the $t$-structure $(\D^{\leq 0},\D^{\geq 0})$; and
    \item $\K=\thick(\P)$ the smallest thick subcategory of $\D$ containing $\P$.
\end{itemize}
For any integer $p$, the inclusions $\D^{\leq p}\to\D$ and $\D^{\geq p}\to\D$ admit right and left adjoints, respectively, denoted by
$$\sigma^{\leq p}:\D\hookrightarrow \D^{\leq p}\text{ and }\sigma^{\geq p}:\D\hookrightarrow \D^{\geq p},$$
which are called \emph{truncations}. For any integers $p\leq q$, we define the functor
$$H^{[p,q]}:=\sigma^{\geq p}\circ\sigma^{\leq q}\simeq\sigma^{\leq q}\circ\sigma^{\geq p}:\D\to \D^{[p,q]}.$$
For any integer $q$, we define the functor
\[H^q:=H^{[q,q]}[q]:\D\to\H.\]
Since the $t$-structure $(\D^{\leq 0}_\P,\D^{\geq 0}_\P)$ is non-degenerate, we have
\[\D^{\leq 0}_\P=\{\YY\mid H^i(\YY)=0,\ \forall i>0\}\text{ and } \D^{\geq 0}_\P=\{\YY\mid H^i(\YY)=0,\ \forall i<0\}.\]

By definition, the presilting subcategory $\P$ is in fact a silting subcategory of $\K$. We recall a well-known correspondence between silting subcategories of $\K$ and $t$-structures on $\D$.

\begin{proposition}\label{silt to t-str}
    Let $\S$ be a silting subcategory of $\K$ such that it is contravariantly finite in $\D$ and $\S\subseteq\P^{[0,d]}$ for some integer $d>0$. Then $(\D_\S^{\leq 0},\D_\S^{\geq 0})$ is a $t$-structure on $\D$ and $\D^{\leq -d}\subseteq\D_\S^{\leq 0}\subseteq\D^{\leq 0}$. Moreover, if two silting subcategories $\S$ and $\S'$ of $\K$ satisfy $\D_\S^{\leq 0}=\D_{\S'}^{\leq 0}$, then $\S=\S'$.
\end{proposition}

\begin{proof}
    The first statement follows essentially from \cite[Lemma~5.1]{AMY}, while the last statement follows from $\D_\S^{\leq 0}\cap\K=\K_\S^{\leq 0}$ and applying Lemma~\ref{sil=} to the case $\C=\K$.
\end{proof}

From now on, let $d$ be a fixed positive integer, and define the \emph{$d$-extended heart} of $(\D^{\leq 0},\D^{\geq 0})$ as the full subcategory
$$\dH:=\D^{[-d+1,0]}=\D^{\leq 0}\cap\D^{\geq -d+1}=\H[d-1]*\cdots*\H[1]*\H.$$
Since both $\D^{\leq 0}$ and $\D^{\geq -d+1}$ are closed under extensions and direct summands, $\dH$ is closed under them as well. Consequently, $\dH$ is an extriangulated category in the sense of \cite{NP}. The biadditive functor $\mathbb{E}$ is defined by 
$$\mathbb{E}(\MM,\NN):=\Hom(\MM,\NN[1])$$ 
for $\MM,\NN\in \dH$. Given a triangle in $\D$
$$\NN\xrightarrow{f}\LL\xrightarrow{g}\MM\xrightarrow{\delta}\NN[1]$$ 
with $\NN,\LL,\MM\in\dH$, we define
$$\NN\xrightarrow{f}\LL\xrightarrow{g}\MM$$ 
to be an \emph{extriangle} in $\dH$. Unless otherwise stated, every extriangle in this paper is taken in $\dH$. For further details on the theory of extriangulated categories, we refer the reader to \cite{NP}.

In this section, we introduce the notion of AIR tilting subcategories of $\dH$ and establish their correspondence with silting subcategories of $\K$. An object $\SS$ of $\K$ is called \emph{$(d+1)$-term} if $\SS\in\P^{[0,d]}$. Similarly, a subcategory $\S$ of $\K$ is $(d+1)$-term if $\S\subseteq\P^{[0,d]}$.
\begin{definition}
    A \emph{$\P$-presentation} of an object $\MM\in\dH$ is a $(d+1)$-term object $\SS$ of $\K$ satisfying $H^{[-d+1,0]}(\SS)\cong\MM$.

    A \emph{$\P$-presentation} of a subcategory $\M\subseteq\dH$ is a $(d+1)$-term subcategory $\S$ of $\K$ satisfying $\add H^{[-d+1,0]}(\S)=\M$, with $$H^{[-d+1,0]}(\S):=\{H^{[-d+1,0]}(\SS)\mid\SS\in\S\}.$$
\end{definition}

By definition, any $(d+1)$-term subcategory $\S$ of $\K$ is a $\P$-presentation of a subcategory $\add H^{[-d+1,0]}(\S)$ of $\dH$. Conversely, any subcategory of $\dH$ admits a $\P$-presentation.

\begin{lemma}\label{projpres}
    Every object $\MM\in\dH$ admits a $\P$-presentation, and consequently every subcategory $\M\subseteq\dH$ does as well.
\end{lemma}

\begin{proof}
    Let $\MM\in\dH$. Since $\P$ is contravariantly finite in $\D$, there exist triangles
    \[\MM_{i+1}\to P_{i}\xrightarrow{g_i} \MM_{i}\xrightarrow{h_i} \MM_{i+1}[1],\ 0\leq i\leq d,\]
    where $\MM_0=\MM$ and each $g_i$ is a right $\P$-approximation of $\MM_{i}$. Applying $\Hom(P,-)$ for $P\in\P$ to these triangles, we see that all $\MM_i$ ($0\le i\le d$) lie in $\D^{\le 0}$. Extend the composition $h:=h_d[d]\circ\cdots\circ h_1[1]\circ h_0:\MM\to \MM_{d+1}[d+1]$ to a triangle in $\D$:
    $$\MM_{d+1}[d]\to\SS\to \MM\xrightarrow{h} \MM_{d+1}[d+1].$$
    By the octahedral axiom, we have that 
    $$\SS\in P_0*P_1[1]*\cdots*P_d[d]\in\P^{[0,d]}.$$ 
    Since $\MM_{d+1}[d]\in\D^{\leq 0}[d]=\D^{\leq -d}$ and $\MM\in\dH$, we have $\MM\cong H^{[-d+1,0]}(\SS)$. 
    
    Let $\M$ be a subcategory of $\dH$. For any object $\MM\in\M$, take a $\P$-presentation $\SS_\MM$ of $\MM$. The full subcategory $\S$ of $\K$ consisting of $\SS_\MM$ for all $\MM\in\M$ is a $\P$-presentation of $\M$; indeed, it is enough to show that $\S=\add\S$. This follows directly from $\M=\add\M$, since the functor $H^{[-d+1,0]}$ preserves direct sums and direct summands.
\end{proof}

Let $\S$ be a $(d+1)$-term subcategory of $\K$. We define two associated full subcategories of $\dH$ as follows:
$$\T(\S):=\D_\S^{\leq0}\cap\dH=\{\XX\in\dH\mid \Hom(\S,\XX[i])=0,\ \forall i>0\},$$
$$\F(\S):=\D_\S^{\geq0}[-1]\cap\dH=\{\XX\in\dH\mid \Hom(\S,\XX[i])=0,\ \forall i\leq 0\}.$$

We recall the following notion of $s$-torsion pairs from \cite{AET}.

\begin{definition}[{\cite[Definition~3.1]{AET}}]
    A pair of subcategories $(\T,\F)$ of $\dH$ is called an \emph{$s$-torsion pair}, with $\T$ referred to as the \emph{$s$-torsion class}, if the following hold.
    \begin{enumerate}
        \item $\Hom(\T,\F)=0=\Hom(\T,\F[-1])$.
        \item For any $\XX\in\dH$, there is an extriangle
        $\TT\to\XX\to\FF,$
        where $\TT\in\T$ and $\FF\in\F$.
    \end{enumerate}
\end{definition}

\begin{proposition}\label{rmk:silt to tor}
    Let $\S$ be a $(d+1)$-term silting subcategory of $\K$ that is contravariantly finite in $\D$. Then $(\T(\S),\F(\S))$ is an $s$-torsion pair in $\dH$. Moreover, if two $(d+1)$-term silting subcategories $\S$ and $\S'$ of $\K$ that are contravariantly finite in $\D$ satisfy $\T(\S)=\T(\S')$, then $\S=\S'$. 
\end{proposition}

\begin{proof}
    By Proposition~\ref{silt to t-str}, $(\D_\S^{\leq 0},\D_\S^{\geq 0})$ is a $t$-structure on $\D$ with $\D^{\leq 0}[d]\subseteq\D_\S^{\leq 0}\subseteq\D^{\leq 0}$. It then follows from \cite[Proposition~3.12]{AET} that $(\T(\S),\F(\S))$ is an $s$-torsion pair in $\dH$. 
    
    Let $\S$ and $\S'$ be two $(d+1)$-term silting subcategories of $\K$ that are contravariantly finite in $\D$ such that $\T(\S)=\T(\S')$. Then by \cite[Theorem~3.9]{AET}, $\D_\S^{\leq 0}=\D_{\S'}^{\leq 0}$. 
    Hence, by Proposition~\ref{silt to t-str}, we conclude that $\S = \S'$.
\end{proof}

We recall the notion of $s$-factors from \cite{Z}.

\begin{definition}[{\cite[Definition~1.13]{Z}}]
    Let $\X$ be a subcategory of $\dH$, and let $s$ be a positive integer. Suppose there exist extriangles
    \begin{equation}\label{eq:tris}
        \ZZ_i\to \XX_i\to\ZZ_{i-1},\ 1\leq i\leq s,
    \end{equation}
    where $\XX_1, \ldots, \XX_s \in \X$. Then $\ZZ_0$ is called an \emph{$s$-factor} of $\X$. We denote by $\Fac_s(\X)$ the full subcategory of $\dH$ consisting of all $s$-factors of $\X$. For convenience, we set $\Fac_0(\X) := \dH$.
\end{definition}

It follows from the definition that, for any $s \geq 0$, the following inclusion holds:
\begin{equation}\label{eq:subset}
    \X\subseteq\Fac_{s+1}(\X)\subseteq\Fac_s(\X).
\end{equation}
We remark that, in general, $\Fac_s(\X)$ is not necessarily closed under direct summands.

\begin{definition}
    A subcategory $\M$ of $\dH$ is said to be \emph{AIR tilting with respect to its $\P$-presentation} $\S$ if $\S$ is contravariantly finite in $\D$ and $\T(\S)=\Fac_d(\M)$. We say that $\M$ is \emph{AIR tilting} if it admits a $\P$-presentation $\S$ with respect to which it is AIR tilting.
\end{definition}
 
The following is the main result of this section.

\begin{theorem}\label{thm:bi}
    There is a bijection from the set of $(d+1)$-term silting subcategories of $\K$ that are contravariantly finite in $\D$ to the set of AIR tilting subcategories of $\dH$, given by $\S \mapsto \add H^{[-d+1,0]}(\S)$.
\end{theorem}

We first collect several preliminary results needed for the proof of the theorem.

Let $\SS\in\P^{[0,d]}$. By \ref{item:P1}, there exists an integer $m<0$ such that $\P\subseteq\D^{[m,0]}$, which implies $\SS\in\D^{[m-d,0]}$. Applying the truncation functor, we obtain a triangle
\begin{equation}\label{eq:trunP}
\sigma^{\leq -d}(\SS)\to\SS\to H^{[-d+1,0]}(\SS)\to \sigma^{\leq -d}(\SS)[1],
\end{equation}
where $\sigma^{\leq -d}(\SS)\in\D^{[m-d,-d]}$ and $H^{[-d+1,0]}(\SS)\in\D^{[-d+1,0]}=\dH$. Applying the functor $\Hom(-,\NN)$ to this triangle for any $\NN\in\dH$ gives the following lemma.

\begin{lemma}\label{lem2}
    Let $\SS\in\P^{[0,d]}$. For any $\NN \in \dH$, there are functorial isomorphisms
    $$\Hom(H^{[-d+1,0]}(\SS),\NN[i])\cong\Hom(\SS,\NN[i]),\ i\leq 0,$$
    and a functorial injective homomorphism
    $$\Hom(H^{[-d+1,0]}(\SS),\NN[1])\hookrightarrow\Hom(\SS,\NN[1]).$$
\end{lemma}

An object $\MM$ in a subcategory $\X\subseteq \dH$ is called \emph{$\mathbb{E}$-projective} (resp. \emph{$\mathbb{E}$-injective}) if $\mathbb{E}(\MM,\XX)=0$ (resp. $\mathbb{E}(\XX,\MM)=0$) for all $\XX\in\X$.

\begin{lemma}\label{lem:E-proj}
	Let $\S$ be a $(d+1)$-term subcategory of $\K$. Then $\S$ is a presilting subcategory of $\K$ if and only if $\add H^{[-d+1,0]}(\S)\subseteq\T(\S)$. Moreover, in this case, any object of $\add H^{[-d+1,0]}(\S)$ is $\mathbb{E}$-projective in $\T(\S)$.
\end{lemma}

\begin{proof}
    For any two objects $\SS,\SS'\in\S$, applying $\Hom(\SS',-)$ to the triangle~\eqref{eq:trunP}, we obtain, for any integer $i$, an exact sequence
    \[\begin{array}{rl}
        & \Hom(\SS',\sigma^{\leq -d}(\SS)[i]) \to \Hom(\SS',\SS[i]) \to \Hom(\SS',H^{[-d+1,0]}(\SS)[i]) \\
        &\to \Hom(\SS',\sigma^{\leq -d}(\SS)[i+1]).
    \end{array}  \]
    Since $\S\subseteq\P^{[0,d]}$, we have $\D^{\leq -d}\subseteq\D_\S^{\leq 0}$. Therefore, when $i>0$, the first and the last terms of the sequence vanish. It follows that there exist isomorphisms 
    $$\Hom(\SS',\SS[i])\cong \Hom(\SS',H^{[-d+1,0]}(\SS)[i]),\ i>0.$$
    Thus, $\S$ is a presilting subcategory of $\K$ if and only if $H^{[-d+1,0]}(\S)\subseteq\T(\S)$. Since $\T(\S)$ is closed under finite direct sums and direct summands, this is equivalent to $\add H^{[-d+1,0]}(\S)\subseteq\T(\S)$.

    Let $\SS\in\S$. For any $\NN\in\T(\S)$, by definition, we have $\Hom(\SS,\NN[1])=0$. Hence, by the injective homomorphism in Lemma~\ref{lem2}, we have $\Hom(H^{[-d+1,0]}(\SS), \NN[1])=0$. Thus, $H^{[-d+1,0]}(\SS)$ is $\mathbb{E}$-projective in $\T(\S)$. Consequently, every object in $\add H^{[-d+1,0]}(\S)$ is $\mathbb{E}$-projective in $\T(\S)$.
\end{proof}

The following lemma is an analogue of \cite[Lemma~2.5]{Z}, and its proof remains valid in the present setting.

\begin{lemma}\label{lem:fac}
    Let $\S$ be a $(d+1)$-term silting subcategory of $\K$. Then the subcategory $\T(\S)$ of $\dH$ is closed under $d$-factors.
\end{lemma}

We now collect some properties of the $s$-torsion class associated with a silting subcategory.

\begin{proposition}\label{prop3}
	Let $\S$ be a $(d+1)$-term silting subcategory of $\K$ that is contravariantly finite in $\D$. Then the following hold.
    \begin{enumerate}[label=($\alph*$),ref=($\alph*$)]
        \item\label{item:a} For any object $\XX\in\T(\S)$, there exists an extriangle
        $$\ZZ\to H^{[-d+1,0]}(\SS)\to\XX,$$
        where $\SS\in \S$ and $\ZZ\in\T(\S)$.
        \item\label{item:b} $\T(\S)=\Fac_{d}\left(\add H^{[-d+1,0]}(\S)\right)=\Fac_{d+1}\left(\add H^{[-d+1,0]}(\S)\right).$
        \item\label{item:c} $\add H^{[-d+1,0]}(\S)$ coincides with the class of $\mathbb{E}$-projective objects of $\T(\S)$.
    \end{enumerate}
\end{proposition}

\begin{proof}
    To prove \ref{item:a}, since $\S$ is contravariantly finite in $\D$, there exists a right $\S$-approximation $f:\SS\to\XX$. Extend $f$ to a triangle in $\D$:
	$$\YY\to \SS\xrightarrow{f} \XX\to\YY[1].$$
	For any $\SS'\in\S$, apply the functor $\Hom(\SS',-)$ to the triangle. Since both $\SS$ and $\XX$ belong to $\D^{\leq 0}_{\S}$, we deduce that $\YY\in\D^{\leq 0}_{\S}$. By Proposition~\ref{silt to t-str}, it follows that $\YY \in \D^{\leq 0}$.
	
	By the functorial isomorphism in Lemma~\ref{lem2} for $i=0$, the morphism $f$ factors through the middle morphism in the triangle~\eqref{eq:trunP}. Hence, by the octahedral axiom, we obtain a commutative diagram of triangles
	$$\xymatrix{
	&\sigma^{\leq -d}(\SS)\ar@{=}[r]\ar[d]&\sigma^{\leq -d}(\SS)\ar[d]\\
	\XX[-1]\ar[r]\ar@{=}[d]&\YY\ar[r]\ar[d]&\SS\ar[r]^f\ar[d]&\XX\ar@{=}[d]\\
	\XX[-1]\ar[r]&\ZZ\ar[r]\ar[d]&H^{[-d+1,0]}(\SS)\ar[r]^{\qquad f_0}\ar[d]&\XX\\
	&\sigma^{\leq -d}(\SS)[1]\ar@{=}[r]&\sigma^{\leq -d}(\SS)[1]
	}$$
	Considering the triangle in the third row, it suffices to show that $\ZZ\in\T(\S)$. Since $f$ is a right $\S$-approximation of $\XX$, for any $\SS'\in\S$, the map $\Hom(\SS',f)$ is surjective, which implies that the map $\Hom(\SS',f_0)$ is also surjective. Hence, applying $\Hom(\SS',-)$ to the triangle in the third row and using that both $H^{[-d+1,0]}(\SS)$ and $\XX$ lie in $\T(\S)\subseteq\D^{\leq 0}_{\S}$, we deduce that $\ZZ\in\D^{\leq 0}_{\S}$. It remains to show that $\ZZ \in \dH$. From the triangle in the third row, since $\XX[-1] \in \D^{[-d+2, 1]}$ and $H^{[-d+1,0]}(\SS) \in \D^{[-d+1, 0]}$, it follows that $\ZZ \in \D^{[-d+1, 1]}$. On the other hand, from the triangle in the second column, where $\YY \in \D^{\leq 0}$ and $\sigma^{\leq -d}(\SS)[1]\in\D^{\leq-d-1}$, we deduce $\ZZ \in \D^{\leq 0}$. Combining these yields $\ZZ \in \dH$.

    For \ref{item:b}, applying assertion \ref{item:a} repeatedly yields the inclusion
    \begin{equation}\label{eq:inc1}
        \T(\S) \subseteq \Fac_{d+1}\left(\add H^{[-d+1,0]}(\S)\right).
    \end{equation}
    By Lemma~\ref{lem:E-proj}, we know that $\add H^{[-d+1,0]}(\S) \subseteq \T(\S)$. Since by Lemma~\ref{lem:fac}, $\T(\S)$ is closed under $d$-factors, we have
    \begin{equation}\label{eq:inc2}
    \Fac_d\left(\add H^{[-d+1,0]}(\S)\right) \subseteq \T(\S).
    \end{equation}
    Combining the inclusions \eqref{eq:inc1} and \eqref{eq:inc2}, along with the inclusion~\eqref{eq:subset}, yields the equalities claimed in \ref{item:b}.

    For \ref{item:c}, by Lemma~\ref{lem:E-proj}, it suffices to show that any $\mathbb{E}$-projective object $\XX$ of $\T(\S)$ lies in $\add H^{[-d+1,0]}(\S)$. This follows directly from the extriangle in \ref{item:a}, which splits in this case.
\end{proof}

\begin{remark}\label{rmk:enoughproj}
    By applying Proposition~\ref{prop3} to the case where $\S=\P$, we obtain that $\add H^{[-d+1,0]}(\P)$ coincides with the class of $\mathbb{E}$-projective objects of $\dH$, and that $\dH$ has enough projectives in the sense of \cite[Definition~3.25]{NP}. If, in addition, $\P\subseteq\dH$, then $\P$ itself is the class of $\mathbb{E}$-projective objects of $\dH$.
\end{remark}

To establish a relationship between silting subcategories of $\K$ and AIR tilting objects of $\dH$, we need the following lemma.

\begin{lemma}\label{prop2}
    Let $\S$ be a $(d+1)$-term presilting subcategory of $\K$ that is contravariantly finite in $\D$. If for any $\NN\in\T(\S)$, there exists an extriangle
    $$\ZZ \to \MM \to \NN,$$
    where $\MM \in \add H^{[-d+1,0]}(\S)$, then $\S$ is a silting subcategory of $\K$.
\end{lemma}

\begin{proof}
    By applying Proposition~\ref{prop1} to the case $\C=\K$, we see that it suffices to show that $\D_\S^{\leq 0}\cap \K\subseteq \D^{\leq 0}$. We argue by contradiction. Suppose that there exists an object $\XX \in \D_\S^{\leq 0} \cap \K$ such that $\XX \notin \D^{\leq 0}$. Since $\XX \in \K$, by Lemma~\ref{AI2.23}, there exists a positive integer $n$ such that $\XX\in\P^{[-n,n]}\subseteq\D^{\leq n}$. Hence, there exists a positive integer $m$ such that $H^m(\XX) \neq 0$ and $H^i(\XX) = 0$ for all $i > m$. Let $\NN = H^{[-d+1,0]}(\XX[m])\in\dH$, and consider the truncation of $\XX[m]$. Then there exists a triangle in $\D$:
    $$\sigma^{\leq -d}(\XX[m])\to \XX[m]\to \NN\to \sigma^{\leq -d}(\XX[m])[1].$$
    Since $\XX\in\D_\S^{\leq 0}$, we have $\XX[m]\in\D_\S^{\leq -1}$. Since $\S\subseteq\P^{[0,d]}$, we have $\D^{\leq -d}\subseteq\D_\S^{\leq 0}$ and hence $\sigma^{\leq -d}(\XX[m])[1]\in\D_\S^{\leq -1}$. It follows that, by the above triangle, $\NN\in\D_\S^{\leq -1}$. Therefore, $\NN \in \T(\S)$, and by Lemma~\ref{lem2}, we have $\Hom(H^{[-d+1,0]}(\S), \NN) = 0$. Now, by assumption, there exists a triangle in $\D$:
    $$\ZZ\to\MM\to\NN\to\ZZ[1],$$
    where $\MM \in \add H^{[-d+1,0]}(\S)$ and $\ZZ \in \dH$. Then the middle morphism is zero, implying that $\NN$ is a direct summand of $\ZZ[1]$, and hence $H^0(\NN) = 0$. However, $H^0(\NN) = H^m(\XX) \neq 0$, yielding a contradiction. Therefore, $\D_\S^{\leq 0}\cap \K\subseteq \D^{\leq 0}$.
\end{proof}

\begin{proposition}\label{prop:equiv}
    Let $\M\subseteq\dH$ and $\S$ be a $\P$-presentation of $\M$ that is contravariantly finite in $\D$. Then $\S$ is a silting subcategory of $\K$ if and only if $\M$ is an AIR tilting subcategory of $\dH$ with respect to $\S$.
\end{proposition}

\begin{proof}
    The ``only if" part follows directly from Proposition~\ref{prop3}\ref{item:b}. For the ``if" part, since $\M\subseteq\Fac_d(\M)=\T(\S)$ by Lemma~\ref{lem:E-proj}, $\S$ is a presilting subcategory of $\K$. For any object $\NN\in\T(\S)=\Fac_d(\M)$, there exists an extriangle
    $$\ZZ\to\MM\to\NN,$$
    with $\MM\in\M$ and $\ZZ\in\Fac_{d-1}(\M)$. Hence, by Lemma~\ref{prop2}, $\S$ is silting.
\end{proof}

We are now ready to prove Theorem~\ref{thm:bi}.

\begin{proof}[Proof of Theorem~\ref{thm:bi}]
    Let $\S$ be a $(d+1)$-term silting subcategory of $\K$ that is contravariantly finite in $\D$. By definition, we have that $\S$ is a $\P$-presentation of $\add H^{[-d+1,0]}(\S)$. Then, by Proposition~\ref{prop:equiv}, $\add H^{[-d+1,0]}(\S)$ is an AIR tilting subcategory of $\dH$. Therefore, the map is well-defined.

    To prove the injectivity of the map, suppose that there exist two $(d+1)$-term silting subcategories $\S$ and $\S'$ of $\K$ that are contravariantly finite in $\D$ and such that $\add H^{[-d+1,0]}(\S) = \add H^{[-d+1,0]}(\S')$. Then, by Proposition~\ref{prop3}\ref{item:b}, we obtain \[\T(\S)=\Fac_d\left(\add H^{[-d+1,0]}(\S)\right)=\Fac_d\left(\add H^{[-d+1,0]}(\S')\right)=\T(\S').\] Therefore, by Proposition~\ref{rmk:silt to tor}, it follows that $\S=\S'$.
    
    Finally, the surjectivity of the map follows immediately from Proposition~\ref{prop:equiv}.
\end{proof}

An immediate consequence of the theorem is the uniqueness of $\P$-presentations of AIR tilting subcategories.

\begin{corollary}\label{S=S'}
    Let $\M$ be a subcategory of $\dH$. If $\M$ is AIR tilting with respect to both its $\P$-presentations $\S$ and $\S'$, then $\S=\S'$.
\end{corollary}

\begin{proof}
    By Proposition~\ref{prop:equiv}, both $\S$ and $\S'$ are $(d+1)$-term silting subcategories of $\K$ that are contravariantly finite in $\D$ and such that $\add H^{[-d+1,0]}(\S)=\M=\add H^{[-d+1,0]}(\S')$. Therefore, Theorem~\ref{thm:bi} implies $\S=\S'$.
\end{proof}

For any subcategory $\M$ of $\dH$, we define the following full subcategory of $\dH$:
$$\M^{\bot_{\leq0}}:=\{\NN\in\dH\mid\Hom(\M,\NN[i])=0,\ \forall i\leq 0 \}.$$
As a consequence of Theorem~\ref{thm:bi}, we have the following result.

\begin{corollary}\label{cor:injmap}
    There is an injective map $\M\mapsto (\Fac_d(\M),\M^{\bot_{\leq0}})$ from the set of AIR tilting subcategories of $\dH$ to the set of $s$-torsion pairs in $\dH$.
\end{corollary}

\begin{proof}
    By Theorem~\ref{thm:bi} and Proposition~\ref{rmk:silt to tor}, it suffices to show that for any AIR tilting subcategory $\M$ of $\dH$ with respect to its $\P$-presentation $\S$, we have $\T(\S)=\Fac_d(\M)$ and $\F(\S)=\M^{\bot_{\leq0}}$. The first equality follows from Proposition~\ref{prop3}\ref{item:b}, and the second follows from the isomorphisms in Lemma~\ref{lem2}.
\end{proof}

\section{Quasi-tilting subcategories of extended hearts}\label{sec:3}

In this section, we introduce the notion of quasi-tilting subcategories of the extended heart $\dH$, which generalizes AIR-tilting subcategories, and we study their basic properties.

\begin{definition}\label{def:quasitilting}
    A subcategory $\M$ of $\dH$ is called \emph{quasi-tilting} if
    \begin{enumerate}[label=(QT\arabic*), ref=(QT\arabic*)]
        \item\label{item:QT1} every object of $\M$ is $\mathbb{E}$-projective in $\Fac_d(\M)$, and 
        \item\label{item:QT2} $\Fac_d(\M)=\Fac_{d+1}(\M)$.
    \end{enumerate}
\end{definition}

\begin{proposition}\label{silisquasitil}
    Every AIR tilting subcategory of $\dH$ is quasi-tilting. 
\end{proposition}

\begin{proof}
    By Theorem~\ref{thm:bi}, there exists a $(d+1)$-term silting subcategory $\S$ of $\K$ such that $\M=\add H^{[-d+1,0]}(\S)$ and $\T(\S)=\Fac_d(\M)$. The proposition then follows immediately from Proposition~\ref{prop3}.
\end{proof}

We now present the main result of this section, which summarizes several properties of $\Fac_d(\M)$ associated with a quasi-tilting subcategory $\M$. Observe that, by definition, $\Fac_d(\M)$ is closed under finite direct sums.

\begin{theorem}\label{thm:qtilt}
    Let $\M$ be a quasi-tilting subcategory of $\dH$. Then the subcategory $\Fac_d(\M)$ is closed under the following operations:
    \begin{itemize}
        \item extensions,
        \item $d$-factors, and
        \item direct summands.
    \end{itemize} 
    Moreover, $\M$ consists of the $\mathbb{E}$-projective objects in $\Fac_d(\M)$, and $\Fac_d(\M)$ has enough projectives.
\end{theorem}

Throughout the rest of this section, $\M$ denotes a quasi-tilting subcategory of $\dH$. We prove the theorem in several steps.

\begin{lemma}\label{lem1}
    For any $\ZZ\in\Fac_d(\M)$, there exists a triangle in $\D$
    \begin{equation}\label{eq:tri1}
        \ZZ'\to\MM'\xrightarrow{f_{\ZZ}}\ZZ\to\ZZ'[1],
    \end{equation}
    where $f_{\ZZ}$ is a right $\M$-approximation of $\ZZ$, and $\ZZ'\in\Fac_d(\M)$.
\end{lemma}

\begin{proof}
    By \ref{item:QT2}, $\ZZ\in\Fac_d(\M)=\Fac_{d+1}(\M)$. Hence, there exists a triangle
    \[\ZZ'\to\MM'\xrightarrow{f_{\ZZ}}\ZZ\to\ZZ'[1],\]
    where $\ZZ'\in\Fac_d(\M)$ and $\MM'\in\M$. Applying $\Hom(\MM,-)$ for any $\MM\in\M$ to this triangle yields an exact sequence
    $$\Hom(\MM,\MM')\xrightarrow{\Hom(\MM,f_{\ZZ})}\Hom(\MM,\ZZ)\to\mathbb{E}(\MM,\ZZ'),$$
    where the last term vanishes by \ref{item:QT1}. Hence, the map $\Hom(\MM,f_{\ZZ})$ is surjective. Thus, $f_{\ZZ}$ is a right $\M$-approximation of $\ZZ$.
\end{proof}

Consequently, $\Fac_d(\M)$ has enough projectives, and $\M$ coincides with the class of $\mathbb{E}$-projective objects in $\Fac_d(\M)$.

The following Schanuel-type lemma for triangulated categories is taken from \cite[Lemma~2.4]{W2}, in which a sketch of the proof is provided. For completeness, we give a full proof here.

\begin{lemma}\label{lem:Scha}
    Let $\X$ be a full subcategory of an arbitrary triangulated category. Suppose that there are two triangles
    $$\YY\to\XX\xrightarrow{f}\ZZ\to\YY[1],$$
    $$\YY'\to\XX'\xrightarrow{f'}\ZZ\to\YY'[1],$$
    and both $f$ and $f'$ are right $\X$-approximations of $\ZZ$. Then $\YY\oplus\XX'\cong\YY'\oplus\XX$.
\end{lemma}

\begin{proof}
    Since $f'$ is a right $\X$-approximation of $\ZZ$ and $\XX\in\X$, it follows that $f$ factors through $f'$. Hence by the octahedral axiom, there exists a commutative diagram of triangles
    $$\xymatrix{
    \WW\ar@{=}[r]\ar[d]&\WW\ar[d]\\
    \YY\ar[r]\ar[d]&\XX\ar[r]^{f}\ar[d]&\ZZ\ar@{=}[d]\ar[r]^{g}&\YY[1]\ar[d]\\
    \YY'\ar[d]\ar[r]&\XX'\ar[r]^{f'}\ar[d]&\ZZ\ar[r]&\YY'\\
    \WW[1]\ar@{=}[r]&\WW[1]
    }$$
    It follows that there is a triangle
    \[\YY\to\XX\oplus\YY'\to \XX'\xrightarrow{g\circ f'}\YY[1].\]
    Since $f$ is a right $\X$-approximation of $\ZZ$, there exists a morphism $h:\XX'\to \XX$ such that $f'=f\circ h$. Hence, $g\circ f'=g\circ f\circ h=0$. Consequently, the triangle splits, which implies the desired isomorphism $\YY\oplus\XX'\cong\YY'\oplus\XX$.
\end{proof}

The proof of the following lemma is adapted from the proof of \cite[Proposition~2.4]{W2}, with additional details provided.

\begin{lemma}\label{lem4}
    Let $f$ be a right $\M$-approximation of an object $\ZZ\in\Fac_d(\M)$, and let
    $$\ZZ''\to\MM''\xrightarrow{f}\ZZ\to\ZZ''[1]$$
    be a triangle containing $f$ as one of its morphisms. Then $\ZZ''\in\Fac_d(\M)$.
\end{lemma}

\begin{proof}
    By Lemma~\ref{lem1}, there exists a triangle 
    \[\ZZ'\to\MM'\xrightarrow{f_{\ZZ}}\ZZ\to\ZZ'[1]\]
    given in \eqref{eq:tri1}. Then, by Lemma~\ref{lem:Scha}, we have the isomorphism $\ZZ'' \oplus \MM' \cong \ZZ' \oplus \MM''$. Since $\ZZ' \oplus \MM'' \in \Fac_d(\M)$, we apply Lemma~\ref{lem1} again to obtain a triangle
    $$\YY\to\NN\to\ZZ''\oplus\MM'\to\YY[1],$$
    where $\YY\in\Fac_d(\M)$ and $\NN\in\M$. By the octahedral axiom, one obtains a commutative diagram of triangles
    $$\xymatrix{
    &\ZZ''[-1]\ar@{=}[r]\ar[d] &\ZZ''[-1]\ar[d]^{0}\\
    \YY\ar@{=}[d]\ar[r]&\XX\ar[r]\ar[d]&\MM'\ar[r]^{h}\ar[d]&\YY[1]\ar@{=}[d]\\
    \YY\ar[r]&\NN\ar[r]\ar[d]&\ZZ''\oplus\MM'\ar[r]\ar[d]&\YY[1]\\
    &\ZZ''\ar@{=}[r]&\ZZ''
    }$$
    Since $\MM'\in\M$, it follows from \ref{item:QT1} that $\MM'$ is $\mathbb{E}$-projective in $\Fac_d(\M)$. As $\YY\in\Fac_d(\M)$, we deduce that $h=0$. Therefore, $\XX\cong\YY\oplus\MM'\in \Fac_d(\M)$. Consequently, considering the second-column triangle in the diagram, we deduce $\ZZ''\in\Fac_{d+1}(\M)\subseteq\Fac_d(\M)$.
\end{proof}

Now, we complete the proof of Theorem~\ref{thm:qtilt} through the following three propositions. These propositions extend \cite[Propositions~3.3 and 3.7]{W3} and \cite[Proposition~2.2]{W1} from module categories to extended hearts, with their proofs suitably adapted.

\begin{proposition}\label{prop001}
    Let $\XX\to\YY\to\ZZ$ be an extriangle in $\dH$. Then the following statements hold.
    \begin{enumerate}
        \item If both $\XX$ and $\YY$ belong to $\Fac_d(\M)$, then $\ZZ\in\Fac_s(\M)$ for any $s\geq 0$.
        \item If both $\XX$ and $\ZZ$ belong to $\Fac_d(\M)$, then $\YY\in\Fac_s(\M)$ for any $s\geq 0$.
        \item If both $\YY$ and $\ZZ$ belong to $\Fac_d(\M)$ and $\mathbb{E}(\M,\XX)=0$, then $\XX\in\Fac_s(\M)$ for any $s\geq 0$.
    \end{enumerate}
    In particular, $\Fac_d(\M)$ is closed under extensions.
\end{proposition}

\begin{proof}
    Denote by
    $$\XX\xrightarrow{\alpha}\YY\xrightarrow{\beta}\ZZ\xrightarrow{\gamma}\XX[1]$$
    the triangle in $\D$ corresponding to the extriangle $\XX\to\YY\to\ZZ$. 
    
    (1) We proceed by induction on $s$. The starting case $s=0$ is trivial. Assume the assertion holds for $s=l$, and consider $s=l+1$.
    
    By Lemma~\ref{lem1}, there exist two triangles
    $$\XX'\to \MM'_{\XX}\xrightarrow{f_{\XX}} \XX\to\XX'[1],\ \YY'\to \MM'_{\YY}\xrightarrow{f_{\YY}} \YY\to\YY'[1],$$
    where $f_{\XX}$ and $f_{\YY}$ are right $\M$-approximations of $\XX$ and $\YY$, respectively, and $\XX',\YY'\in\Fac_d(\M)$. Starting with the commutative diagram
    $$\xymatrix{
    \MM'_{\XX}\ar[r]^{\left(\begin{smallmatrix}1\\0\end{smallmatrix}\right)\quad\ }\ar[d]_{f_{\XX}}&\MM'_{\XX}\oplus\MM'_{\YY}\ar[d]^{\delta:=\left(\begin{smallmatrix}
        \alpha\circ f_{\XX} & f_{\YY}
    \end{smallmatrix}\right)}\\
    \XX\ar[r]^{\alpha}&\YY
    }$$
    and by the $3\times 3$ Lemma, one obtains a diagram of triangles
    $$\xymatrix{
    \XX[-1]\ar[r]\ar[d]&\YY[-1]\ar[r]\ar[d]&\ZZ[-1]\ar[r]\ar[d]&\XX\ar[d]\\
    \XX'\ar[r]\ar[d]&\YY''\ar[r]\ar[d]&\ZZ''\ar[d]\ar[r]&\XX'[1]\ar[d]\\
    \MM'_{\XX}\ar[r]^{\left(\begin{smallmatrix}1\\0\end{smallmatrix}\right)\quad\ }\ar[d]_{f_{\XX}}&\MM'_{\XX}\oplus\MM'_{\YY}\ar[d]^{\delta}\ar[r]^{\quad\left(\begin{smallmatrix}0&1\end{smallmatrix}\right)}&\MM'_{\YY}\ar[d]\ar[r]^{0}&\MM'_{\XX}[1]\ar[d]\\
    \XX\ar[r]^{\alpha}&\YY\ar[r]^{\beta}&\ZZ\ar[r]^{\gamma}&\XX[1]
    }$$
    Since $f_{\YY}$ is a right $\M$-approximation of $\YY$, the morphism $\delta$ is also a right $\M$-approximation of $\YY$. Therefore, by Lemma~\ref{lem4}, we have $\YY''\in\Fac_d\left(\M\right)$. Considering the triangles in the second row and the third column of the above diagram, it follows that $\ZZ''\in\dH$. Therefore, applying the induction hypothesis to the triangle in the second row yields $\ZZ'' \in \Fac_{l}\left(\M\right)$. Hence, by the triangle in the third column, we conclude that $\ZZ \in \Fac_{l+1}\left(\M\right)$, completing the induction step.

    (2) We still use induction on $s$. The starting case $s=0$ is trivial. Assume that the assertion holds when $s = l$ for some $l \geq 0$. Let us consider the case $s = l + 1$.
    
    By Lemma~\ref{lem1}, there exist triangles
    $$\XX'\to \MM'_{\XX}\xrightarrow{f_{\XX}} \XX\to\XX'[1],\ \ZZ'\to \MM'_{\ZZ}\xrightarrow{f_{\ZZ}} \ZZ\to\ZZ'[1],$$
    where $f_{\XX}$ and $f_{\ZZ}$ are right $\M$-approximations of $\XX$ and $\ZZ$, respectively, and $\XX', \ZZ' \in \Fac_d\left(\M\right)$. Since $\XX \in \Fac_d(\M)$ and $\MM'_{\ZZ} \in \M$, by \ref{item:QT1}, $\mathbb{E}(\MM'_{\ZZ}, \XX) = 0$. Hence, $\gamma \circ f_{\ZZ} = 0$, and there is a commutative diagram
    $$\xymatrix{
    \MM'_{\ZZ}\ar[r]^{0\quad}\ar[d]_{f_{\ZZ}}&\MM'_{\XX}[1]\ar[d]^{f_{\XX}[1]}\\
    \ZZ\ar[r]_{\gamma\quad}&\XX[1]
    }$$
    Applying the $3 \times 3$ Lemma to this provides the following diagram of triangles:
    $$\xymatrix{
    \XX[-1]\ar[r]\ar[d]&\YY[-1]\ar[r]\ar[d]&\ZZ[-1]\ar[r]\ar[d]&\XX\ar[d]\\
    \XX'\ar[r]\ar[d]&\YY''\ar[r]\ar[d]&\ZZ'\ar[d]\ar[r]&\XX'[1]\ar[d]\\
    \MM'_{\XX}\ar[r]\ar[d]_{f_{\XX}}&\MM'_{\XX}\oplus\MM'_{\ZZ}\ar[d]\ar[r]&\MM'_{\ZZ}\ar[d]^{f_{\ZZ}}\ar[r]^{0\quad}&\MM'_{\XX}[1]\ar[d]^{f_{\XX}[1]}\\
    \XX\ar[r]^{\alpha}&\YY\ar[r]^{\beta}&\ZZ\ar[r]^{\gamma\quad}&\XX[1]
    }$$
    Considering the triangle in the second row of the diagram, we see that $\YY'' \in \dH$. Applying the induction hypothesis to this triangle, we obtain $\YY'' \in \Fac_l(\M)$. Hence, by the triangle in the second column, we conclude that $\YY\in\Fac_{l+1}(\M)$, completing the induction step.

    (3) By Lemma~\ref{lem1}, there exists a triangle
    $$\YY'\to \MM'_{\YY}\xrightarrow{f_{\YY}} \YY\to\YY'[1],$$
    where $f_{\YY}$ is a right $\M$-approximation of $\YY$ and $\YY'\in\Fac_d(\M)$. By taking the composition $\beta\circ f_{\YY}$ and applying the octahedral axiom, we obtain a commutative diagram of triangles
    $$\xymatrix{
    &\YY'\ar@{=}[r]\ar[d]&\YY'\ar[d]\\
    \ZZ[-1]\ar[r]\ar@{=}[d]&\NN\ar[r]\ar[d]&\MM'_{\YY}\ar[r]^{\ \beta\circ f_{\YY}}\ar[d]^{f_{\YY}}&\ZZ\ar@{=}[d]\\
    \ZZ[-1]\ar[r]&\XX\ar[r]^{\alpha}\ar[d]&\YY\ar[r]^{\beta}\ar[d]&\ZZ\\
    &\YY'[1]\ar@{=}[r]&\YY'[1]
    }$$
    By \ref{item:QT1}, we have $\mathbb{E}(\M, \YY') = 0$. By assumption, $\mathbb{E}(\M, \XX) = 0$. Therefore, by the triangle in the second column of the above diagram, it follows that $\mathbb{E}(\M, \NN) = 0$. It follows that $\beta\circ f_{\YY}$ is a right $\M$-approximation of $\ZZ$. Hence, by Lemma~\ref{lem4}, we have $\NN\in\Fac_d(\M)$. Thus, by applying part (1) of this proposition to the triangle in the second column of the diagram, we obtain $\XX \in \Fac_s(\M)$ for any $s \geq 0$.
\end{proof}

\begin{proposition}
    Any direct summand of an object of $\Fac_d(\M)$ belongs to $\Fac_s(\M)$ for any $s\geq 0$. In particular, $\Fac_d(\M)$ is closed under direct summands.
\end{proposition}

\begin{proof}
    We proceed by induction on $s \geq 0$, with the starting case $s = 0$ being trivial. Assume that the assertion holds for $s = l$, where $l \geq 0$ is an integer. We consider the case $s = l + 1$.

    Let $\ZZ$ be an object of $\Fac_d(\M)$ such that $\ZZ=\XX\oplus\YY$. By Lemma~\ref{lem1}, there exists a triangle
    $$\ZZ'\to\MM'\xrightarrow{f_{\ZZ}}\ZZ\to\ZZ'[1]$$
    with $\MM'\in\M$ and $\ZZ'\in\Fac_d(\M)$. Composing $f_\ZZ$ with the canonical projection from $\ZZ$ to $\XX$ and applying the octahedral axiom, we obtain a commutative diagram of triangles:
    $$\xymatrix{
    &\ZZ'\ar[d]\ar@{=}[r]&\ZZ'\ar[d]\\
    \XX[-1]\ar[r]\ar@{=}[d]&\LL\ar[r]\ar[d]&\MM'\ar[d]_{f_{\ZZ}}\ar[r]&\XX\ar@{=}[d]\\
    \XX[-1]\ar[r]^{0}&\YY\ar[r]\ar[d]&\ZZ\ar[d]\ar[r]&\XX\\
    &\ZZ'[1]\ar@{=}[r]&\ZZ'[1]
    }$$
    By adding $\XX$ to two terms in the triangle in the second column, we have a triangle
    $$\ZZ'\to\LL\oplus\XX\to \YY\oplus\XX\to \ZZ'[1].$$
    Since $\Fac_d(\M)$ is closed under extensions by Proposition~\ref{prop001}, and both $\ZZ'$ and $\YY \oplus \XX = \ZZ$ belong to $\Fac_d(\M)$, it follows that $\LL \oplus \XX \in \Fac_d(\M)$. Thus, by the induction assumption, we have $\LL\in \Fac_l(\M)$. Therefore, the triangle in the second row of the above diagram shows that $\XX\in \Fac_{l+1}(\M)$.
\end{proof}

\begin{proposition}\label{closefac}
    For any integer $s \geq 0$, we have 
    $$\Fac_s(\Fac_d(\M)) \subseteq \Fac_s(\M).$$
    In particular, $\Fac_d(\M)$ is closed under $d$-factors.
\end{proposition}

\begin{proof}
    We proceed by induction on $s$, with the trivial starting case $s=0$. Suppose the inclusion holds when $s=l$ for some $l\geq 0$. We consider the case $s=l+1$.
    For any $\ZZ\in\Fac_s(\Fac_d(\M))$, by definition, there exists a triangle
    $$\ZZ'\to\YY\xrightarrow{f}\ZZ\to\ZZ'[1],$$
    where $\YY\in\Fac_d(\M)$ and $\ZZ'\in\Fac_{s-1}(\Fac_d(\M))$. Then by Lemma~\ref{lem1}, there exists a triangle
    $$\YY'\to\MM'\xrightarrow{f_{\YY}}\YY\to\YY'[1],$$
    with $\MM'\in\M$ and $\YY'\in\Fac_d(\M)$. By taking the composition $f\circ f_{\YY}:\MM'\to\ZZ$ and applying the octahedral axiom, we have the following commutative diagram of triangles:
    \begin{equation}\label{eq:001}
    \xymatrix{
    &\YY'\ar@{=}[r]\ar[d]&\YY'\ar[d]\\
    \ZZ[-1]\ar[r]\ar@{=}[d]&\LL\ar[r]\ar[d]_{g}&\MM'\ar[r]^{f\circ f_{\YY}}\ar[d]_{f_{\YY}}&\ZZ\ar@{=}[d]\\
    \ZZ[-1]\ar[r]&\ZZ'\ar[r]\ar[d]&\YY\ar[r]^{f}\ar[d]&\ZZ\\
    &\YY'[1]\ar@{=}[r]&\YY'[1]
    }
    \end{equation}
    Considering the triangle in the second column, we have $\LL\in\dH$. 
    Hence, the triangle in the second row implies that $\ZZ\in\Fac_1(\M)$. Thus, the inclusion holds for $s=l+1$ when $l=0$.
    
    In what follows, assume $l\geq 1$. Since $\ZZ'\in\Fac_{l}(\Fac_d(\M))$, by the induction assumption, we have $\ZZ'\in\Fac_{l}(\M)$. Hence, by definition, there exists a triangle
    $$\widetilde{\ZZ'}\to\widetilde{\MM}\to\ZZ'\xrightarrow{h}\widetilde{\ZZ'}[1],$$
    with $\widetilde{\MM}\in\M$ and $\widetilde{\ZZ'}\in\Fac_{l-1}(\M)=\Fac_{l-1}(\Fac_d(\M))$. By taking the composition $h\circ g:\LL\to \widetilde{\ZZ'}[1]$ and applying the octahedral axiom, there is a commutative diagram of triangles:
    \begin{equation}\label{eq:002}
    \xymatrix{
    &\widetilde{\ZZ'}\ar@{=}[r]\ar[d]&\widetilde{\ZZ'}\ar[d]\\
    \YY'\ar@{=}[d]\ar[r]&\NN\ar[r]\ar[d]&\widetilde{\MM}\ar[d]\ar[r]&\YY'[1]\ar@{=}[d]\\
    \YY'\ar[r]&\LL\ar[r]^{g}\ar[d]_{h\circ g}&\ZZ'\ar[d]^h\ar[r]&\YY'[1]\\
    &\widetilde{\ZZ'}[1]\ar@{=}[r]&\widetilde{\ZZ'}[1]
    }
    \end{equation}
    Since $\Fac_d(\M)$ is closed under extensions by Proposition~\ref{prop001}, and both $\YY' \in \Fac_d(\M)$ and $\widetilde{\MM} \in \M \subseteq \Fac_d(\M)$, the triangle in the second row of diagram~\eqref{eq:002} implies that $\NN \in \Fac_d(\M)$. Then by the triangle in the second column of diagram~\eqref{eq:002}, we have $\LL\in\Fac_l(\Fac_d(\M))$. Hence, by the induction hypothesis, we have $\LL \in \Fac_l(\M)$. It then follows from the triangle in the second row of diagram~\eqref{eq:001} that $\ZZ \in \Fac_{l+1}(\M)$, which completes the proof.
\end{proof}

\section{Tilting subcategories of extended hearts}\label{sec:tilt}

Throughout this section, we assume that $\P\subseteq\dH$. By Remark~\ref{rmk:enoughproj}, the extriangulated category $\dH$ has enough projectives $\P$. Thus, following \cite[Section~5.1]{LN}, one can define higher $\mathbb{E}$-extensions in $\dH$ as follows. For any object $\MM\in\dH$, there exist extriangles
\begin{equation}\label{eq:projreso}
    \ZZ_i\xrightarrow{f_i} P_{i-1}\xrightarrow{g_{i-1}} \ZZ_{i-1},\ i>0,
\end{equation}
with $\ZZ_0=\MM$ and $P_{i-1} \in \P$. For any $\NN\in\dH$, define 
$$\mathbb{E}^i(\MM,\NN):=\mathbb{E}(\ZZ_{i-1},\NN).$$ 
Alternatively, the higher $\mathbb{E}$-extensions can be defined directly as
$$\mathbb{E}^i(\MM,\NN):=\Hom(\MM,\NN[i]),$$
for any $\MM,\NN\in\dH$ and $i>0$. Indeed, applying $\Hom(-,\NN)$ to the extriangles in \eqref{eq:projreso}, we obtain an exact sequence for any $i>0$ and $j>0$:
$$\Hom(P_{i-1},\NN[j])\to \Hom(\ZZ_i,\NN[j])\to\Hom(\ZZ_{i-1},\NN[j+1])\to\Hom(P_{i-1},\NN[j+1]),$$
where the first and last terms vanish. Hence, there are functorial isomorphisms
$$\Hom(\ZZ_i,\NN[j])\cong\Hom(\ZZ_{i-1},\NN[j+1]),\ i>0,\ j>0,$$
which imply that
$\mathbb{E}^i(\MM,\NN)\cong\Hom(\MM,\NN[i])$ functorially in both variables.

An object $\MM \in \dH$ is said to have \emph{projective dimension at most $m$} if there exist extriangles of the form \eqref{eq:projreso} such that $\ZZ_m \in \P$. Then $\MM$ has projective dimension at most $m$ if and only if $\mathbb{E}^j(\MM,\NN)=0$ for all $j>m$ and all $\NN\in\dH$; see, e.g., \cite[Lemma~3]{ZZ}.

\begin{definition}\label{def:tilt}
    A subcategory $\M$ of $\dH$ is called \emph{tilting} if the following hold.
    \begin{enumerate}[label=(T\arabic*), ref=(T\arabic*)]
        \item\label{item:T1} Every object of $\M$ has projective dimension at most $d$.
        \item\label{item:T2} $\mathbb{E}^i(\M,\M)=0$ for all $i>0$.
        \item\label{item:T3} For any $P\in\P$, there exist extriangles
        $$\YY_{i-1}\to \MM_{i-1}\to \YY_i,\ 1\leq i\leq r,$$
        for some $r$, with $\YY_0=P$ and $\MM_0,\dots,\MM_{r-1},\YY_r\in\M$.
    \end{enumerate}
\end{definition}

We note that a notion of tilting objects in the extended heart also appears in the recent work \cite{AMMP}; however, the tilting objects considered there are required to have projective dimension at most one.

Analogously to the case of module categories, a subcategory of $\dH$ with finite projective dimension can be regarded as a subcategory of $\K$.

\begin{lemma}\label{projdim}
    Let $\M$ be a subcategory of $\dH$ and $m$ a positive integer. Then $\M$ has projective dimension at most $m$ if and only if $\M\subseteq\P^{[0,m]}$. 
\end{lemma}

\begin{proof}
    If every object in $\M$ has projective dimension at most $m$, then by definition, any object $\MM\in\M$ admits extriangles as in \eqref{eq:projreso}, which imply that
    $$\M\subseteq (\add P_0)*(\add P_1[1])*\cdots*(\add P_{m-1}[m-1])*(\add Z_m[m])\subseteq\P^{[0,m]}.$$

    Conversely, suppose that $\M\subseteq\P^{[0,m]}$. Then for any object $\MM\in\M$, there exists a triangle
    $$\MM'\to P'\to \MM\to\MM'[1],$$
    where $P'\in\P\subseteq\dH$ and $\MM'\in\P^{[0,m-1]}\subseteq\D^{\leq 0}$. It follows that $\MM'\in \M[-1]* \dH\subseteq \D^{\geq -d+1}$. Hence, we have $\MM'\in\D^{\leq 0}\cap\D^{\geq -d+1}=\dH$. Therefore, we may apply induction on $m$ to deduce that $\MM$ has projective dimension at most $m$, where in the starting case $m=0$, $\MM\in\P\subseteq\dH$ has projective dimension zero.
\end{proof}

We have the following equivalent characterization of tilting subcategories of $\dH$.

\begin{proposition}\label{tilt=silt}
    Let $\M$ be a subcategory of $\dH$. Then $\M$ is a tilting subcategory of $\dH$ if and only if it is a $(d+1)$-term silting subcategory of $\K$.
\end{proposition}

\begin{proof}
    By Lemma~\ref{projdim}, $\M$ satisfies condition \ref{item:T1}, if and only if $\M\subseteq\P^{[0,d]}$. In this case, $\M$ satisfies \ref{item:T2} and \ref{item:T3} if and only if $\M$ is a silting subcategory of $\K$. This completes the proof.
\end{proof}

A subcategory $\M$ of $\dH$ is called \emph{$\mathbb{E}$-universal} if for any object $\XX\in\dH$, there exists an extriangle
$$\XX\to\YY\to \MM,$$
with $\MM\in\M$ such that for any $\MM'\in\M$, the induced map $\Hom(\MM',\MM)\to\mathbb{E}(\MM',\XX)$ is surjective; see \cite[Definition~5.7]{AMMP}. Since the extriangulated structure of $\dH$ is induced from the triangulated structure of $\D$, a subcategory $\M$ of $\dH$ is $\mathbb{E}$-universal if and only if for any object $\XX\in\dH$, there exists a right $\M$-approximation of $\XX[1]$.

\begin{lemma}\label{confin}
    Let $\M$ be a subcategory of $\dH$, satisfying $\M\subseteq\P^{[0,d]}$. Then the following are equivalent.
    \begin{enumerate}
        \item $\M$ is both contravariantly finite and $\mathbb{E}$-universal in $\dH$.
        \item $\M$ is contravariantly finite in $\dplusH$.
        \item $\M$ is contravariantly finite in $\D$.
    \end{enumerate}
\end{lemma}

\begin{proof}
    (1)$\implies$(2): For any $\XX\in\dplusH$, by truncation, there exists a triangle
    \[H^{-d}(\XX)[d]\xrightarrow{\alpha}\XX\xrightarrow{\beta} H^{[-d+1,0]}(\XX)\xrightarrow{\gamma} H^{-d}(\XX)[d+1].\]
    Since $\M$ is contravariantly finite in $\dH$, there exists a right $\M$-approximation $f:\MM\to H^{[-d+1,0]}(\XX)$. Since $\M\subseteq\P^{[0,d]}$, we have $\Hom(\M,H^{-d}(\XX)[d+1])=0$, which implies that $\gamma\circ f=0$. Hence, there exists $h:\MM\to\XX$ such that $f=\beta\circ h$. Since $\M$ is $\mathbb{E}$-universal in $\dH$, there exists a right $\M$-approximation $g:\MM'\to H^{-d}(\XX)[d]$. See the following diagram.
    \[\xymatrix{
    \MM'\ar[d]_{g}&\MM''\ar[d]_{\varphi}\ar@{-->}[r]^{\varphi'}\ar@{-->}[dl]_{\varphi''}\ar@{-->}[l]_{\varphi'''}&\MM\ar[d]^{f}\ar[dl]_{h}\\
    H^{-d}(\XX)[d]\ar[r]^{\qquad\alpha}&\XX\ar[r]^{\beta\qquad}& H^{[-d+1,0]}(\XX)\ar[r]^{\gamma\quad}& H^{-d}(\XX)[d+1]
    }\]
    It suffices to show that the morphism
    \[\left(\begin{matrix}
        h&\alpha\circ g
    \end{matrix}\right)\colon \MM\oplus\MM'\to \XX\]
    is a right $\M$-approximation of $\XX$. Indeed, for any morphism $\varphi:\MM''\to \XX$, since $f$ is a right $\M$-approximation of $H^{[-d+1,0]}(\XX)$, the composition $\beta\circ \varphi:\MM''\to H^{[-d+1,0]}(\XX)$ factors through $f$, that is, there exists $\varphi':\MM''\to \MM$ such that $\beta\circ \varphi=f\circ\varphi'$. Then $\beta\circ(\varphi-h\circ\varphi')=0$. It follows that there exists $\varphi'':\MM''\to H^{-d}(\XX)[d]$ such that $\varphi-h\circ\varphi'=\alpha\circ\varphi''$. Since $g$ is a right $\M$-approximation of $H^{-d}(\XX)[d]$, there exists $\varphi''':\MM''\to\MM'$ such that $\varphi''=g\circ\varphi'''$. Hence, $\varphi=h\circ\varphi'+\alpha\circ g\circ\varphi'''$. Thus, we have $$\varphi=\left(\begin{matrix}
        h&\alpha\circ g
    \end{matrix}\right)\circ\left(\begin{matrix}
        \varphi'\\\varphi'''
    \end{matrix}\right),$$
    which shows that the morphism $\left(\begin{matrix}
        h&\alpha\circ g
    \end{matrix}\right)$ is a right $\M$-approximation, as required.

    (2)$\implies$(3): For any $\XX\in\D$, by truncation, there exist two triangles
    \[\sigma^{\leq -d-1}(\XX)\to \sigma^{\leq 0}(\XX)\xrightarrow{g} H^{[-d,0]}(\XX)\to\sigma^{\leq -d-1}(\XX)[1],\]
    \[\sigma^{>0}(\XX)[-1]\to \sigma^{\leq 0}(\XX)\xrightarrow{h}\XX\to \sigma^{>0}(\XX).\]
    Since $\M\subseteq\P^{[0,d]}$, for any $\MM\in\M$, we have $$\Hom(\MM,\sigma^{\leq -d-1}(\XX))=0=\Hom(\MM,\sigma^{\leq -d-1}(\XX)[1]),$$
    $$\Hom(\MM,\sigma^{>0}(\XX)[-1])=0=\Hom(\MM,\sigma^{>0}(\XX)).$$
    Hence, applying $\Hom(\MM,-)$ to the above triangles, yields functorial isomorphisms
    \[\Hom(\MM,g)\colon\Hom(\MM,\sigma^{\leq 0}(\XX))\xrightarrow{
    \cong
    }\Hom(\MM,H^{[-d,0]}(\XX)),\]
    \[\Hom(\MM,h)\colon\Hom(\MM,\sigma^{\leq 0}(\XX))\xrightarrow{
    \cong
    }\Hom(\MM,\XX).\]
    It follows from the first functorial isomorphism that, for any right $\M$-approximation $f$ of $H^{[-d,0]}(\XX)$ whose existence is due to (2), there exists a right $\M$-approximation $f'$ of $\sigma^{\leq 0}(\XX)$ such that $f=g\circ f'$. Then it follows from the second functorial isomorphism that $h\circ f'$ is a right $\M$-approximation of $\XX$.

    (3)$\implies$(1): Trivial.
\end{proof}

For any subcategory $\M$ of $\dH$, we define the following full subcategory of $\dH$:
$$\T(\M):=\{\XX\in\dH\mid\mathbb{E}^i(\M,\XX)=0\ \text{for all } i>0\}.$$
The equivalence between (1) and (2) in the following result establishes a connection between tilting subcategories and AIR tilting subcategories of extended hearts. The equivalence between (1) and (4) shows that a contravariantly finite $\mathbb{E}$-universal subcategory $\M$ of $\dH$ is tilting in our sense if and only if $\M$ is a tilting subcategory of projective dimension $d$ of the extriangulated category $\dH$ in the sense of \cite[Definition~4.1]{ZZ}. 

\begin{proposition}\label{prop5}
    Let $\M$ be a contravariantly finite $\mathbb{E}$-universal subcategory of $\dH$. Then the following are equivalent.
    \begin{enumerate}
        \item $\M$ is a tilting subcategory of $\dH$.
        \item $\M$ is its own $\P$-presentation and is AIR tilting with respect to itself.
        \item $\M\subseteq\P^{[0,d]}$ and $\T(\M)=\Fac_d(\M)=\Fac_{d+1}(\M)$.
        \item $\M\subseteq\P^{[0,d]}\cap\T(\M)$, and for any $\XX\in\T(\M)$, there exists an extriangle
        $\ZZ\to\MM\to\XX,$
        with $\MM\in\M$ and $\ZZ\in\T(\M)$.
    \end{enumerate}
\end{proposition}

\begin{proof}
    (1) $\implies$ (2): By condition~\ref{item:T1} in Definition~\ref{def:tilt} and Lemma~\ref{projdim}, we have $\M\in\P^{[0,d]}$. It follows that $\M$ is its own $\P$-presentation. Moreover, by Proposition~\ref{tilt=silt} and Lemma~\ref{confin}, $\M$ is a silting subcategory of $\K$ that is contravariantly finite in $\D$. Therefore, by Proposition~\ref{prop:equiv}, $\M$ is AIR tilting with respect to itself.

    (2) $\implies$ (3): From (2) and Proposition~\ref{prop:equiv}, it follows that $\M$ is a $(d+1)$-silting subcategory of $\K$. Then by Proposition~\ref{prop3}, the statement in (3) follows.
    
    (3) $\implies$ (4): Note that $\M\subseteq\Fac_d(\M)=\T(\M)$. For any $\XX\in\T(\M)=\Fac_{d+1}(\M)$, by definition, there exists an extriangle
    $$\ZZ\to\MM\to\XX,$$
    with $\MM\in\M$ and $\ZZ\in\Fac_d(\M)=\T(\M)$.

    (4) $\implies$ (1): By Lemma~\ref{prop2}, $\M$ is a silting subcategory of $\K$. Hence, by Proposition~\ref{tilt=silt}, $\M$ is a tilting subcategory of $\dH$.
\end{proof}

Throughout the rest of this section, let $\I$ be the full subcategory of $\dH$ consisting of $\mathbb{E}$-injective objects of $\dH$ and we assume that $\dH$ has enough injectives.

\begin{lemma}\label{X=0}
    For any object $\XX\in\dH$, if $\Hom(\XX,\I[i])=0$ for any $-d+1\leq i\leq 0$, then $\XX=0$.
\end{lemma}

\begin{proof}
    We first claim that for any object $\YY\in\dH$, if $\Hom(\YY,\I)=0$ then $\YY[1]\in\dH$. Indeed, since $\dH$ has enough injectives $\mathcal{I}$, there exists a triangle
    \[\YY\xrightarrow{f} I\to \YY'\to\YY[1],\]
    with $I\in\I$ and $\YY'\in\dH$. Since $f=0$, we have that $\YY[1]$ is a direct summand of $\YY'$ and hence belongs to $\dH$. Thus, by induction, we deduce that $\XX[i]\in\dH$ for any $0\leq i\leq d$. However, this is not possible unless $\XX=0$.
\end{proof}

It is well known that one characterization of a classical tilting module is that every injective module lies in the induced torsion class. We now investigate the meaning of this characterization for tilting subcategories of $\dH$, starting with the following lemma. 

\begin{lemma}\label{lem00}
    Let $\S\subseteq\P^{[0,d]}$. If $\I\subseteq\T(\S)$, then $\S\subseteq\dH$.
\end{lemma}

\begin{proof}
    Recall from \eqref{eq:trunP}, there exists a triangle
    $$\sigma^{\leq -d}(\SS)\to\SS\to H^{[-d+1,0]}(\SS)\to \sigma^{\leq -d}(\SS)[1],$$
    where $\MM\in\dH$ and $\sigma^{\leq -d}(\SS)\in\D^{[-d+m,-d]}$ for some negative integer $m$. We claim $\sigma^{\leq -d}(\SS)=0$, which completes the proof. Indeed, if this is not the case, applying $\Hom(-,I)$ for any $I\in\I$ to this triangle, we obtain exact sequences
    $$\Hom(\SS,I[i])\to \Hom(\sigma^{\leq -d}(\SS),I[i])\to\Hom(\MM,I[i+1]),\ i>0,$$
    where the last term vanishes since $I$ is $\mathbb{E}$-injective in $\dH$. Moreover, as $I \in\mathcal{I}\subseteq \T(\S)$, the first term also vanishes. Hence, the middle term $\Hom(\sigma^{\leq -d}(\SS),I[i])=0$ for any $i>0$. Since $\sigma^{\leq -d}(\SS)\in\D^{[-d+m,-d]}$, there exists a negative integer $l$ with $m \leq l\leq 0$, such that $H^{-d+l}(\sigma^{\leq -d}(\SS))\neq 0$ and $H^{q}(\sigma^{\leq -d}(\SS))=0$ for any $q>-d+l$. Let $X=\left(\sigma^{\leq -d}(\SS)\right)[-d+l]$. Then $\XX\in\D^{\leq 0}$ and $H^0(\XX)=H^{-d+l}(\sigma^{\leq -d}(\SS))\neq 0$. By truncation of $\XX$, there exists a triangle
    \[\sigma^{\leq -d}(\XX)\to\XX\to H^{[-d+1,0]}(\XX)\to\sigma^{\leq -d}(\XX)[1].\]
    Applying $\Hom(-,I)$ to it, we obtain exact sequences
    \[\Hom(\sigma^{\leq -d}(\XX)[1],I[i])\to \Hom(H^{[-d+1,0]}(\XX),I[i])\to \Hom(X,I[i]),\ -d+1\leq i\leq 0,\]
    where the first term is zero since $\sigma^{\leq -d}(\XX)[1]\in\D^{\leq -d-1}$ and $I[i]\in\D^{\geq -d+1}$, and the last term is zero since $\Hom(\sigma^{\leq -d}(\SS),I[i])=0$ for any $i>0$. Therefore, $\Hom(H^{[-d+1,0]}(\XX),I[i])=0$ for any $-d+1\leq i\leq 0$, which by Lemma~\ref{X=0}, we have $H^{[-d+1,0]}(\XX)=0$, a contradiction. Thus, we have $\sigma^{\leq -d}(\SS)=0$.
\end{proof}

In what follows, we establish an alternative connection between tilting subcategories and AIR tilting subcategories that avoids the use of $\P$-presentations.

\begin{proposition}\label{prop002}
    Let $\M$ be a contravariantly finite $\mathbb{E}$-universal AIR tilting subcategory of $\dH$. Then $\M$ is a tilting subcategory if and only if $\I\subseteq\Fac_d(\M)$.
\end{proposition}

\begin{proof}
    For the ``only if" part, since $\M$ is tilting, by Proposition~\ref{prop5}, $\Fac_d(\M)=\T(\M)$, which contains $\I$, as every object in $\I$ is $\mathbb{E}$-injective in $\dH$.

    For the ``if" part, by definition, there exists a $\P$-presentation $\S$ of $\M$ such that $\M$ is AIR tilting with respect to $\S$. Then, by Proposition~\ref{prop:equiv}, $\S$ is silting. Hence, by Proposition~\ref{prop3}, we have $\T(\S)=\Fac_d(\M)$. Thus, $\I\subseteq\T(\S)$. By Lemma~\ref{lem00}, it follows that $\S\subseteq\dH$, which implies $\S=\M$. Thus, by Proposition~\ref{tilt=silt}, $\M$ is a tilting subcategory of $\dH$. 
\end{proof}

The main result of this section provides an equivalent characterization of tilting subcategories and establishes a connection between tilting subcategories and quasi-tilting subcategories. We note that, in view of Proposition~\ref{prop5}, the equivalence between (1) and (2) has in fact already been established in a more general setting in \cite{ZZ}; see Theorem 1 and Remark 3 therein.

\begin{theorem}\label{thm:equiv}
    For any contravariantly finite $\mathbb{E}$-universal subcategory $\M$ of $\dH$, the following statements are equivalent.
    \begin{itemize}
        \item[(1)] $\M$ is a tilting subcategory of $\dH$.
        \item[(2)] $\T(\M)=\Fac_d(\M)$.
        \item[(3)] $\M$ is a quasi-tilting subcategory of $\dH$ satisfying $\I\subseteq\Fac_d(\M)$.
    \end{itemize}
\end{theorem}

\begin{proof}
    (1) $\implies$ (2): This follows directly from Proposition~\ref{prop5}.

    (2) $\implies$ (3): First, since any object of $\I$ is $\mathbb{E}$-injective in $\dH$, we have $\I\subseteq\T(\M)=\Fac_d(\M)$. Thus, it remains to show that $\M$ is quasi-tilting.

    Since $\Fac_d(\M)=\T(\M)$, any object of $\M$ is $\mathbb{E}$-projective in $\Fac_d(\M)$, that is, \ref{item:QT1} holds for $\M$. To verify \ref{item:QT2}, let $\XX\in\Fac_d(\M)$. By definition, there exists an extriangle
    $$\XX'\to \MM\xrightarrow{g}\XX,$$
    with $\MM\in\M$ and $\XX'\in\Fac_{d-1}(\M)$. Since $\M$ is contravariantly finite in $\dH$, there is a right $\M$-approximation $f:\MM'\to \XX$ of $\XX$. Hence, there exists a morphism $h:\MM\to \MM'$ such that $g=f\circ h$. It follows that, by the octahedral axiom, there exists a commutative diagram of triangles:
    $$\xymatrix{
    &\YY[-1]\ar@{=}[r]\ar[d]&\YY[-1]\ar[d]\\
    \XX[-1]\ar[r]\ar@{=}[d]&\XX'\ar[r]\ar[d]&\MM\ar[d]^{h}\ar[r]^{g}&\XX\ar@{=}[d]\\
    \XX[-1]\ar[r]&\ZZ\ar[r]\ar[d]&\MM'\ar[r]^f\ar[d]&\XX\\
    &\YY\ar@{=}[r]&\YY
    }$$
    Therefore, there exists a triangle 
    $$\XX'\to\MM\oplus\ZZ\to\MM'\to\XX'[1],$$
    in which both $\XX'$ and $\MM'$ belong to $\dH$. Hence, $\MM\oplus\ZZ$ also lies in $\dH$, and thus $\ZZ$ lies in $\dH$ as well. Applying $\Hom(\MM'',-)$ for any $\MM''$ to the triangle in the third row of the above diagram and using that $f$ is a right $\M$-approximation of $\XX$, we deduce $\ZZ\in\T(\M)=\Fac_d(\M)$, which implies $\XX\in\Fac_{d+1}(\M)$. Thus, $\Fac_d(\M)\subseteq\Fac_{d+1}(\M)$. Combining this with the inclusion~\eqref{eq:subset}, we conclude that \ref{item:QT2} holds for $\M$. Hence, $\M$ is a quasi-tilting subcategory of $\dH$.

    (3) $\implies$ (1): For any $\XX\in\dH$, since $\dH$ has enough injectives $\I$, there exist extriangles
    \begin{equation}\label{eq:tri003}
        \YY_{i-1}\to I_{i}\to \YY_i,\ i\geq 1,
    \end{equation}
    where $\YY_0=\XX$ and $I_i\in\I$ for all $i\geq 1$. Since $\I\subseteq\Fac_d(\M)$, by definition, $\YY_i\in\Fac_d(\Fac_d(\M))$ for all $i\geq d$. Then, by Proposition~\ref{closefac}, we have $\YY_i\in\Fac_d(\M)$ for any $i\geq d$. By \ref{item:QT1}, it follows that $\mathbb{E}(\M,\YY_i)=0$ for all $i\geq d$. Applying $\Hom(\MM,-)$ for any $\MM\in\M$ to the extriangles in \eqref{eq:tri003}, we obtain isomorphisms
    \begin{equation}\label{eq:isos}
        \Hom(\MM,\YY_i[j])\cong \Hom(\MM,\YY_{i-1}[j+1])\cong\cdots\cong \Hom(\MM,\XX[i+j]),\ i\geq 1,j\geq 0.
    \end{equation}
    Hence, $\mathbb{E}^j(\MM,\XX)\cong\Hom(\MM,\XX[j])\cong\Hom(\MM,\YY_{j-1}[1])=0$ for all $j>d$. Therefore, the projective dimension of $\MM$ is at most $d$. Then, by Lemma~\ref{projdim}, $\M\subseteq\P^{[0,d]}$. Hence, by Proposition~\ref{prop5} and \ref{item:QT2}, it suffices to show that $\T(\M)=\Fac_d(\M)$. 
    
    If the object $\XX$ in the first paragraph belongs to $\T(\M)$, 
    then by the isomorphisms in \eqref{eq:isos}, we have $\YY_i\in\T(\M)$ for all $i\geq 1$. Since $\YY_d\in\Fac_d(\M)$,  applying Proposition~\ref{prop001}~(3) repeatedly yields $\YY_{d-1},\ldots,\YY_1,\YY_0=\XX \in\Fac_d(\M)$. Therefore, $\T(\M)\subseteq\Fac_d(\M)$.
    
    If the object $\XX$ in the first paragraph belongs to $\Fac_d(\M)$, then by applying Proposition~\ref{prop001}~(1) to the extriangles in \eqref{eq:tri003}, we have $\YY_i\in\Fac_d(\M)$ for all $i>0$. By \ref{item:QT1}, it follows that $\mathbb{E}(\M,\YY_i)=0$ for all $i>0$. Then, by the isomorphisms in \eqref{eq:isos}, we have $\mathbb{E}^j(\MM,\XX) \cong \Hom(\MM,\XX[j]) \cong \Hom(\MM,\YY_{j-1}[1])=0$ for any $\MM\in\M$ and all $j\geq 1$. It follows that $\XX\in \T(\M)$, which implies that $\Fac_d(\M)\subseteq\T(\M)$. Thus, we obtain $\T(\M)=\Fac_d(\M)$ as required.
\end{proof}

Combining Corollary~\ref{cor:injmap}, Proposition~\ref{prop5} and Theorem~\ref{thm:equiv}, we obtain the following corollary.

\begin{corollary}\label{cor:injmap2}
    There is an injective map $\M\mapsto \Fac_d(\M)$ from the set of contravariantly finite $\mathbb{E}$-universal tilting subcategories of $\dH$ to the set of $s$-torsion classes in $\dH$ that contains all injective objects.
\end{corollary}

\section{Applications}\label{sec:app}

In this section, we apply our results to the following two settings.
\begin{enumerate}
    \item[(I)] $\D$ is a Hom-finite Krull-Schmidt $\k$-linear triangulated category, where $\k$ is a field.
    \item[(II)] $\D$ is a triangulated category with arbitrary coproducts.
\end{enumerate}
Note that in Case (II), the extended heart $\dH$ is closed under arbitrary coproducts. Each setting includes a concrete subcase:
\begin{itemize}
    \item $\D=D^b(\mod A)$, the bounded derived category of finitely generated modules over a finite-dimensional $\k$-algebra $A$, with $\P=\add A$, $\H\simeq\mod A$, and $\K\simeq K^b(\proj A)$. In this case, $\dH=\emod{d}{A}$, the category of \emph{$d$-extended (small) modules} over $A$. Accordingly, we refer to this case as the $\emod{d}{A}$ case.
    \item $\D=D(\lmod R)$, the derived category of modules over a unitary ring $R$, where $\P=\ladd R$, $\H\simeq\lmod R$, and $\K\simeq K^b(\lproj R)$. In this setting, $\dH=\elmod{d}{R}$, the category of \emph{$d$-extended large modules} over $R$. Accordingly, we refer to this case as the $\elmod{d}{R}$ case.
\end{itemize}

For any object $\XX$ of $\D$, we set
\[\A(\XX)=\begin{cases}
    \add\XX & \text{in Case (I),}\\
    \ladd\XX & \text{in Case (II),}
\end{cases}\]
where $\ladd\XX$ denotes the \emph{large additive closure} of $\XX$ in $\D$, that is, the full subcategory of $\D$ consisting of all direct summands of (possibly infinite) coproducts of copies of $\XX$. A common feature of these two settings is the following lemma.

\begin{lemma}\label{approx}
    Let $\XX$ be an object of $\D$. Then $\A(\XX)$ is contravariantly finite in $\D$.
\end{lemma}

\begin{proof}
    Let $\YY$ be an object of $\D$. Let $I$ be a basis of $\Hom(\XX,\YY)$ in Case (I); in Case (II), let $I=\Hom(\XX,\YY)$. Then the object $\XX^{(I)}$ belongs to $\A(\XX)$. Consider the morphism
    $$\XX^{(I)}\xrightarrow{(g)_{g\in I}}\YY,$$
    whose components are given by the elements of $I$. By construction, this morphism is a right $\A(\XX)$-approximation of $\YY$, proving that $\A(\XX)$ is contravariantly finite.
\end{proof}

Hence, when considering subcategories of the form $\A(\XX)$ for some object $\XX$, the hypotheses concerning contravariantly finiteness required in the results of the preceding sections are no longer needed. In particular, we can define silting objects of $\K$ and AIR tilting objects of $\dH$ as follows.

\begin{definition}
    An object $\SS$ of $\K$ is called \emph{silting} if 
    \begin{enumerate}
        \item[(S1)] $\Hom(\A(\SS),\A(\SS)[i])=0$ for all $i>0$, and
        \item[(S2)] $\thick(\A(\SS))=\K$.
    \end{enumerate}
    Two silting objects $\SS$ and $\SS'$ are called \emph{equivalent} if $\A(\SS)=\A(\SS')$.
    
    An object $\MM$ of $\dH$ is called \emph{AIR tilting} if there exists $\SS\in\P^{[0,d]}$ such that 
    $$H^{[-d+1,0]}(\SS)\cong\MM\text{ and }\T(\PP)=\Fac_d(\A(\MM)).$$
    Two AIR tilting objects $\MM$ and $\MM'$ are called \emph{equivalent} if $\A(\MM)=\A(\MM')$.
\end{definition}

In the case $\K=K^b(\proj A)$, a silting object $\PP$ of $\K$ is (up to isomorphism) a usual silting complex. In the case $\K=K^b(\lproj R)$, a silting object of $\K$ is (up to isomorphism) a semi-tilting complex in the sense of \cite{W1}.

In the $\emod{d}{A}$ case, comparing Proposition~\ref{prop:equiv} with \cite[Theorem~4.7]{Z}, an object $\MM$ is AIR tilting if and only if there exists $P\in\proj A$ such that the pair $(\MM,P)$ is a $\AR{d}$-tilting pair in the sense of \cite[Definition~4.5]{Z}. In particular, when $d=1$, this coincides with the notion of support $\tau$-tilting modules introduced in \cite{AIR}. In the $\elmod{d}{R}$ case, when $d=1$, this coincides with the notion of silting modules introduced in \cite{AMV}.

Applying Theorem~\ref{thm:bi} and Corollary~\ref{cor:injmap}, we obtain the following result.

\begin{corollary}
    There is a bijection between the set of equivalence classes of $(d+1)$-term silting objects of $\K$ and the set of equivalence classes of AIR tilting objects of $\dH$, given by $\SS \mapsto H^{[-d+1,0]}(\SS)$. Moreover, there is an injective map from the set of equivalence classes of AIR tilting objects of $\dH$ to the set of $s$-torsion classes in $\dH$, given by $\MM\mapsto\Fac_d(\A(\MM))$.
\end{corollary}

In the $\emod{d}{A}$ case, this recovers part of \cite[Theorem~4.7]{Z} (cf. also \cite[Theorem~4.1]{G}). In the $\elmod{d}{R}$ case, when $d=1$, this recovers part of \cite[Theorem~4.11]{AMV}.

\begin{definition}
    An object $\MM$ of $\dH$ is called \emph{quasi-tilting} if 
    \begin{enumerate}
        \item[(QT1)] $\MM$ is $\mathbb{E}$-projective in $\Fac_d(\A(\MM))$, and
        \item[(QT2)] $\Fac_d(\A(\MM))=\Fac_{d+1}(\A(\MM))$.
    \end{enumerate}
\end{definition}

In the $\emod{d}{A}$ and $\elmod{d}{R}$ cases, when $d=1$, this notion coincides with that of quasi-tilting modules introduced in \cite{CDT,AMV}; see \cite[Lemma~3.1]{AMV}. By Proposition~\ref{silisquasitil}, any quasi-tilting object of $\dH$ is AIR tilting. The converse has been established in \cite{W4} for the $\emod{d}{A}$ case when $d=1$.

We also want to remark that, in the $\emod{d}{A}$ case, an object $\MM\in\emod{d}{A}$ satisfying (QT1) coincides with what is called a \emph{positive $\AR{d}$-rigid} object in the sense of \cite[Definition~4.1]{Z}; see \cite[Remark~4.3]{Z}.

Since by construction, $\Fac_d(\A(\MM))$ is closed under finite direct sums in Case (I), and under arbitrary coproducts in Case (II), an application of Theorem~\ref{thm:qtilt} yields the following result.

\begin{corollary}
    Let $\MM$ be a quasi-tilting object of $\dH$. Then $\Fac_d(\A(\MM))$ is closed under extensions and $d$-factors, has enough projectives given by $\A(\MM)$, and satisfies $$\Fac_d(\A(\MM))=\A(\Fac_d(\A(\MM))).$$
\end{corollary}

In the remainder of the paper, we restrict our attention to the cases: $\dH=\emod{d}{A}$ and $\dH=\elmod{d}{R}$. In this setting, we refer to the objects of $\dH$ as \emph{extended modules}. Note that in both cases we have the inclusions $\P\subseteq\H\subseteq\dH$. Consequently, all notions and results developed in Section~\ref{sec:tilt} apply directly to these two cases.

\begin{definition}
    An extended module $\MM\in\dH$ is called \emph{tilting} if 
    \begin{enumerate}
        \item[(T1)] $\MM$ has projective dimension at most $d$;
        \item[(T2)] $\mathbb{E}^i(\A(\MM),\A(\MM))=0$ for all $i>0$; and 
        \item[(T3)] for any $P\in\P$, there exist extriangles
                    $$\YY_{i-1}\to \MM_{i-1}\to \YY_i,\ 1\leq i\leq r,$$
                    for some $r$, with $\YY_0=P$ and $\MM_0,\dots,\MM_{r-1},\YY_r\in\A(\MM)$.
    \end{enumerate}
\end{definition}

By Proposition~\ref{prop5} and Theorem~\ref{thm:bi}, the tilting $d$-extended modules in $\dH$ are precisely those extended modules that are isomorphic to $(d+1)$-term silting complexes $\PP = (P^i, \partial^i : P^i \to P^{i+1}; -d\leq i\leq 0)$ in $\K$ such that $\partial^{-d}$ is injective.

The following example illustrates the motivation behind the notion of tilting extended modules.

\begin{example}\label{exm:tiltmod}
    A module $T\in\H$ is called $d$-tilting (see \cite{Mi} for the case $\H=\mod A$, and \cite{AC} for $\H=\lmod R$), if the following conditions hold.
    \begin{enumerate}
        \item[(TM1)] The projective dimension of $T$ is at most $d$.
        \item[(TM2)] $\Ext^i(\A(T),\A(T))=0$ for all $i>0$.
        \item[(TM3)] There exists an exact sequence $$0\to A\text{ (or $R$)}\to T_0\to\cdots\to T_{r}\to 0$$ in $\H$ with $T_i\in\A(T)$ for all $0\leq i\leq r$, for some $r$.
    \end{enumerate}
    If we regard $\H$ as a full subcategory of $\dH$ in the natural way, then a module $T\in\H$ is $d$-tilting if and only if it is a tilting extended module in $\dH$.
\end{example}

Any silting complex can be regarded as a tilting object of a suitable extended module category.

\begin{example}
    Let $\SS$ be an arbitrary silting complex in $\K$. Since $\SS$ is bounded, there exist an integer $t$ and a positive integer $\tilde{d}$ such that $H^i(\SS)=0$ for all $i>t$ and $i\leq t-\tilde{d}$. It follows that $\SS[t] \in \tilde{d}\text{-}\H$. Moreover, since $\SS[t]$ can be identified with its own $\P$-presentation, it is a tilting extended module in $\tilde{d}$-$\H$ by Proposition~\ref{prop5}.
\end{example}

Note that in both cases, $\H=\mod A$ and $\H=\lmod R$, the category $\H$ has enough injectives, denoted by $\mathcal{J}$. It follows that $\dH$ has enough injectives $\mathcal{J}[d-1]$. Consequently, we can apply Theorem~\ref{thm:equiv} to obtain the following result, which generalizes \cite[Theorem~3.11]{Ba} and \cite[Theorem~4.3]{W3}.

\begin{corollary}
    Let $\MM$ be an extended module in $\dH$. Then the following statements are equivalent.
    \begin{enumerate}
        \item[(1)] $\MM$ is tilting.
        \item[(2)] $\T(\MM)=\Fac_d(\A(\MM))$.
        \item[(3)] $\MM$ is quasi-tilting and $\mathcal{J}[d-1]\subseteq\Fac_d(\A(\MM))$.
    \end{enumerate}
\end{corollary}

We note one further interesting example satisfying the condition in Case (I): let $\D=D_{fd}(\mod B)$ be the finite-dimensional derived category of a proper non-positive differential graded $\k$-algebra $B$, with $\P = \add B$, $\H \simeq \mod H^0(B)$, and $\K = \per B$. We refer to $\dH$ in this setting as a $d$-truncated subcategory of $D_{fd}(\mod B)$. If one additionally assumes that $\P \subseteq \dH$, then $\dH$ has enough projectives given by $\P$ and enough injectives given by $\mathbb{D}\P[d-1]$. Consequently, all the notions and results developed in this section apply to this case as well. We also note that the Auslander–Reiten theory for $\dH$ has been established in \cite{MP} in this setting, generalizing the corresponding theory for $\emod{d}{A}$ \cite{Z}.

\end{document}